\documentclass{amsart}
\usepackage[latin1]{inputenc}
\usepackage{amssymb}
\usepackage{amsmath}
\usepackage{enumerate}
\usepackage{epsfig,color,xcolor}

\usepackage[all]{xy}
\usepackage{pgf,tikz}
\usepackage{mathrsfs}
\usetikzlibrary{arrows}
\usepackage{cite}
\usepackage{hyperref}
\newtheorem{theorem}{Theorem}[section]
\newtheorem{lemma}[theorem]{Lemma}
\newtheorem{prop}[theorem]{Proposition}
\newtheorem{proposition}[theorem]{Proposition}
\newtheorem{corollary}[theorem]{Corollary}

\newtheorem{definition}[theorem]{Definition}
\newtheorem{example}[theorem]{Example}

\newtheorem{rem}[theorem]{Remark}
\newtheorem{remark}[theorem]{Remark}

\numberwithin{equation}{section}
\newcommand{\rk}{{\rm rank}}

\newcommand{\ra}{\rightarrow}

\newcommand{\Ext}{\mathrm{Ext}}

\newcommand{\PP}{ \mathbb{P}}
\newcommand{\C }{ \mathbb{C}}
\newcommand{\CC }{ \mathbb{C}}

\newcommand{\Z}{\mathbb{Z}}

\newcommand{\cL}{\mathcal{L}}
\newcommand{\cM}{\mathcal{M}}

\newcommand{\Aut}{\mbox{Aut}}
\newcommand{\ZZ}{\mathbb{Z}}

\newcommand{\oo}{\mathcal{O}}
\newcommand{\ii}{\mathcal{I}}

\newcommand{\EE}{\mathscr{E}} 

\newtheoremstyle{dico}
{\baselineskip}   
{\topsep}   
{}  
{0pt}       
{} 
{.}         
{5pt plus 1pt minus 1pt} 
{}          
\theoremstyle{dico}
\newtheorem{say}[theorem]{}

\makeatletter
\def\blfootnote{\xdef\@thefnmark{}\@footnotetext}

\title{Triple covers of K3 surfaces}
\author{Alice Garbagnati and Matteo Penegini}
\date{}
\begin{document}

	
	\maketitle
	\begin{abstract}
We study triple covers of K3 surfaces, following \cite{M85}. We relate the geometry of the covering surfaces with the properties of both the branch locus and the Tschirnhausen vector bundle. In particular, we classify Galois triple covers computing numerical invariants of the covering surface and of its minimal model. We provide examples of non Galois triple covers, both in the case in which the Tschirnhausen bundle splits into the sum of two line bundles and in the case in which it is an indecomposable rank 2 vector bundle. We provide a criterion to construct rank 2 vector bundles on a K3 surface $S$ which determine a non-Galois triple cover of $S$. 

The examples presented are in any admissible Kodaira dimension and in particular we provide the constructions of irregular covers of K3 surfaces and of surfaces with geometrical genus equal to 2 whose transcendental Hodge structure splits in the sum of two Hodge structures of K3 type.    
	\end{abstract}
	\section{Introduction}
The Galois covers of K3 surfaces are a quite classical and interesting argument of research: for example the K3 surfaces which are Galois covers of other K3 surfaces are classified in \cite{X} and the Abelian surfaces which are Galois covers of K3 surfaces are classified in \cite{F}. The study of surfaces with higher Kodaira dimension which are covers of K3 surfaces is less systematic and sporadic examples appear in order to construct specific surfaces see e.g. \cite{CD89, L19,L20, PZ19, RRS19, S06}.
Nevertheless, a systematic approach to the study of the double covers of K3 surfaces is presented in \cite{Gdouble}, where smooth double covers are classified and their birational invariants are given. In the same paper certain covers surfaces with $p_g=2$ are described with more details, since the geometry of these surfaces is quite interesting (see e.g. \cite{L19,L20}).

The aim of this paper is to analyse triple covers of K3 surfaces. One of the main differences between covers of degree 2 and the ones of degree 3 is that the latter are not necessarily Galois.\\

In Section \ref{sec: triple cover in algebraic geometry} we present the general theory of the triple covers of surfaces following \cite{M85, Tan02}. We consider a smooth surface $S$ and a triple cover $f\colon X\ra S$ which is naturally associated to a rank 2 vector bundle  $\EE$, with the property $$f_*\mathcal{O}_X=\mathcal{O}_S\oplus \EE.$$ The vector bundle $\EE$ is called the Tschirnhausen vector bundle of the cover. There are three possibilities:
\begin{itemize}
\item the cover is Galois (in particular $\Z/3\Z$-cyclic); this happens if $\EE$ splits in the direct sum of two line bundles $\mathcal{L}$ and $\mathcal{M}$ which are determined by the non-trivial characters  of $\Z/3\Z$, see Paragraph \ref{say: Galois enigespance}. In this case the triple cover is totally ramified and the singularities of $X$ are due only to singularities of the branch locus; 
\item $\EE$ splits into the direct sum of two line bundles $\mathcal{L}$ and $\mathcal{M}$ but the cover is not Galois. In this case there are components in the branch locus which are of simple, but not total, ramification and we refer to this situation as a \emph{split non-Galois} triple cover; 
\item $\EE$ is indecomposable; also in this case the cover is not Galois and we refer to this case as the \emph{non-split} triple cover. 
\end{itemize}

We are interested in calculating the numerical invariants of the covering surface $X$: $p_g(X)$, $q(X)$, $c_1(X)^2$, $c_2(X)$ and $\kappa(X)$ (which are respectively the geometric genus, the irregularity, the square of the first Chern class, the second Chern class and the Kodaira dimension of the surface $X$). We relate them with the properties of the surface $S$ and of the bundle $\EE$. Some of these numbers are not birational invariants hence, if $X$ is singular, we need to find the minimal model of $X$ to determine the numerical invariants for this model. This aspect is highly non trivial, even if one restricts itself to the Galois triple covers, indeed it requires not just a carefully analysis of the singularities of $X$, but also of the configurations of the $(-1)$-curves appearing in its minimal resolution. 

We restrict to the situation where $S$ is a K3 surface. We provide a criterion to determine the Kodaira dimension of the cover surface $X$ and due to the example constructed in the paper we prove
\begin{theorem} There exist pairs $(S,f)$ such that $f:X\ra S$ is a triple cover either Galois, or split non-Galois or non-split and with either $\kappa(X)=1$ or $\kappa(X)=2$. 
There exist pairs $(S,f)$ such that $f:X\ra S$ is a triple cover either Galois or split non-Galois with $\kappa(X)=0$ and $X$ is necessarily (a possibly singular model of) either a K3 surface or an Abelian surface.
\end{theorem}
Since $X$ is a cover of $S$, it is not possible that $\kappa(X)=-\infty$. We do not know if there exist non split triple cover $f:X\ra S$ with $\kappa(X)=0$.\\

We consider the Galois triple covers of K3 surfaces, obtaining both general results (in Section \ref{sec: the Galois case}) and constructing explicitly families of examples (in Section \ref{sec: examples Galois cover of K3 surfaces}). First, we discuss the singularities of $X$, all coming from the singularities in the branch locus of the cover $f:X\ra S$. There are two different strategies to resolve the singularities of the triple cover: one can blow up $S$ in the singularities of the branch locus until one obtains a birational model of $S$ such that the strict transform of the branch locus is smooth, then one construct the smooth triple cover of this surface. The surface obtained is birational to $X$ and it is called the canonical resolution of $X$ (cfr. \cite{Tan02}). But one can also consider the possibly singular surface $X$ and then resolve its singularities obtaining a resolution which is called minimal resolution of $X$. Note that neither the canonical nor the minimal resolution are necessarily minimal surface. In Sections \ref{say: singu type 1} and \ref{say: sing type 2} we construct both these resolutions if the singularities of the branch locus are ordinary and we observe that they are {\it negligible} (see Definition \ref{def.neg.sing} and Proposition \ref{prop: negligible}) which allows us to compute the numerical invariants not only of $X$, but also of its canonical and of its minimal resolution in Proposition \ref{prop: rito numbers}. Some other singularities in the branch locus are considered in Theorem \ref{say: other singularities in the branch locus} and they are proved to be negligible too. 

Under mild conditions on the smoothness of some components of the branch locus we are often able to identify all the $(-1)$-curves that appear in the resolutions considered and therefore to compute the numerical invariant of the minimal model of $X$, see Proposition \ref{cor: example general type, irreducible D1}, Proposition \ref{prop: case (2) theorem- possibilities}, Proposition \ref{prop: examples case 3 of proposition}.

The main result of this part is a systematic classification of the Galois triple covers of $S$, which can be summarized in the following theorem

\begin{theorem} 
	Let $f\colon X \ra S$ be a normal Galois triple cover of a K3 surface, whose branch locus has $n \geq 1$ connected components, 
	$D_1,\ldots, D_n$ and let $\Lambda_{D_i}$ be the lattice generated by the irreducible components of $D_i$:
	
  $\bullet$ $k(X)=0$ if and only if all the lattices $\Lambda_{D_i}$ are negative definite; in this case $\Lambda_{D_i}\simeq A_2(-1)$, $n=6$ or $n=9$ and $X$ has a trivial canonical bundle.

$\bullet$ $k(X)=2$ if and only if there exists a lattice $\Lambda_{D_i}$  whose signature is $sgn(\Lambda_{D_i})=(1,\rk(\Lambda_{D_i})-1)$; in this case all the others $\Lambda_{D_j}$ are isometric to $A_2(-1)$.

$\bullet$ $k(X)=1$ if and only if there are no lattices $\Lambda_{D_i}$ which are indefinite and there exists at least a lattice $\Lambda_{D_i}$ which is degenerate; in this case the elliptic fibration on $X$ is induced by one on $S$.

In particular if there is a component $D_1$ in the branch locus such that $D_1$ is an irreducible curve, then it holds:

    $\bullet$ if $D_1^2=0$ then $k(X)=1$;

	$\bullet$ if $D_1^2>0$ then $k(X)=2$, $n\leq 10$, $D_1^2=6d$, for an integer $d>0$ and exists an integer $k$ such that $d=n-1+3k$ and	$k\geq -2$. 
	If $D_1$ is smooth and $X^{\circ}$ is the minimal model of $X$ then \begin{equation*}\label{eq: inv Xm special case}\chi(X^{\circ})=5+n+5k,\ K_{X^{\circ}}^2=8n-8+24k, \ e(X^{\circ})=67+5n+36k.\end{equation*}
\end{theorem}

We construct two interesting kinds of examples: in Corollary \ref{cor: pg=2 gen.type} we construct Galois triple covers $f:X\ra S$ of K3 surfaces $S$ which have $p_g(X)=2$. Hence the transcendental Hodge structure of the $X$ are of type $(2,\star,2)$. Since the pullback of the transcendental Hodge structure of $S$ is of K3 type, i.e. of type $(1,\star',1)$, there is a splitting of the transcendental Hodge structure of $X$ in the sum of two Hodge structures of K3 type. One of them is of course geometrically associated to a K3 surface, i.e. to $S$. It would be interest to find another K3 surface associated to the other Hodge structure of K3 type.

The second example of geometric interest is the construction of irregular triples cover of regular surface, see Section \ref{subsec: irregular}. If the Kodaira dimension of the surface $X$ is 0, or 1, the construction is well known: there are triple covers of K3 surfaces with Abelian surfaces (which are irregular and with Kodaira dimension 0); the base change on an elliptically fibered K3 surfaces often produces elliptic fibrations (with Kodaira dimension equal to 1) with a non rational base curve (which forces the surface to be irregular). The situation is more complicated if one requires that $X$ is a surface of general type: such covers exists, but are not very frequent (in the case of double covers classified in \cite{Gdouble} there are very few examples). Here we provide an explicit construction in Theorems \ref{theo: cover S16} and \ref{theo: cover S15}.

In Section \ref{sec: triple split non Galois cover} we briefly discuss the case split non-Galois for tripe cover of K3 surfaces and we provide an example in any admissible Kodaira dimension. The construction are based on the study of the Galois closure.

In Section \ref{sec: triple cover non split} we consider the most general and complicated case, i.e. the case where the vector bundle $\EE$ is indecomposable.
In this case we consider a vector bundle $\EE$ defined by the sequence 
$$0\ra \mathcal{L} \ra \EE^{\vee} \ra \mathcal{M}\otimes \mathcal{I}_Z\ra 0,$$ 
where $\mathcal{L}$ and $\mathcal{M}$ are lines bundles on $S$ and $Z$ a non empty 0-dimensional scheme. Our goal is to list reasonable conditions on $\mathcal{L}$ and $\mathcal{M}$ which assure the existence of the vector bundle $\EE$ and of a triple cover $X\ra S$ whose Tschirnhausen is $\EE$:
\begin{theorem}
Let $\mathcal{L}$ and $\mathcal{M}$ two lines bundle on a K3 surface $S$ such that
$$h^0(S,\mathcal{L}^{\vee}\otimes \mathcal{M})=0,\ h^1(S,\mathcal{L}^{\vee}\otimes \mathcal{M})\geq 1\ h^0(S,\mathcal{L}^{\otimes 2}\otimes \mathcal{M}^{\vee})\geq 1.$$
Let $Z$ be a non empty $0$-dimensional scheme on $S$. Then there exists the triple cover $f:X\ra S$ whose Tschirnhausen bundle is any rank two indecomposable vector bundle $\EE$ obtained by a non-split extension : $$0\ra \mathcal{L} \ra \EE^{\vee} \ra \mathcal{M}\otimes \mathcal{I}_Z\ra 0.$$\end{theorem}

Thanks to this theorem the problem of finding a vector bundle $\EE$ which defines non split triple cover, is reduced to the problem of finding certain line bundles on $S$, with required properties. 
We apply this Theorem to construct a non split triple cover with positive Kodaira dimension and in particular we perform all the computations in one case, obtaining a surface of Kodaira dimension 1, $p_g=6$, $q=3$.\\

	\textbf{Notation and conventions.} We work over the field $\mathbb{C}$
	of complex numbers. 
	
	For $a,b\in \Z$, $a\equiv_n b$ means $a\equiv b\mod n$.
	
	By ``\emph{surface}'' we mean a projective, non-singular surface $S$, and
	for such a surface $\omega_S=\oo_S(K_S)$ denotes the canonical
	class, $p_g(S)=h^0(S, \, \omega_S)$ is the \emph{geometric genus},
	$q(S)=h^1(S, \, \omega_S)$ is the \emph{irregularity} and
	$\chi(\mathcal{O}_S)=1-q(S)+p_g(S)$ is the \emph{Euler-Poincar\'e
		characteristic}. If $q(S)>0$, we call $S$ an \emph{irregular surface}. The minimal model of a surface $S$ will be denoted by $S^{\circ}$; the minimal resolution of the a singular surface $S$ will be denoted by $S'$.
	
	Throughout the paper, we denote Cartier (or Weil) divisors on a variety by
	capital letters and the corresponding line bundles by italic
	letters, so we write for instance $\mathcal{L}=\oo_S(L)$. Moreover, if $d \in H^0(\mathcal{L})$ the corresponding Weil divisor will be denoted by $D$.
	
	Given $Z$ be a purely 0-dimensional subscheme of a variety, we often call $Z$ a $0$-cycle and we denote by $\ell(Z)$ its length.
	
	For a locally free sheaf $\mathcal{F}$ we shall denote its total Chern class by $c(\mathcal{F})$ and its Chen Character by $ch(\mathcal{F})$. 
	
	\medskip
	
	\textbf{Acknowledgments}  Both authors were partially supported by GNSAGA-INdAM. We thank R. Pignatelli and F. Polizzi for useful discussions and suggestions. Both authors wish to thank the referee for many comments and suggestions that improved the presentation of these results.

	\section{Triple covers in algebraic geometry}\label{sec: triple cover in algebraic geometry}
	
	The case of triple covers of algebraic varieties differs sensibly from the double covers case, above of all because the cover might be not Galois. Therefore, a different approach is needed.  This theory of triple covers in algebraic geometry was started by R. Miranda in his seminal paper \cite{M85}, and developed further by Casnati--Ekedahl in \cite{CE96} and Tan in \cite{Tan02}, see also \cite{Pa91, FPV19}. The main result of this theory is the
	following.
	
	\begin{theorem} \emph{\cite[Theorem 1.1]{M85}} \label{teo.miranda}
		A triple cover $f \colon X \to Y$ of an algebraic variety $Y$ is
		determined by a rank $2$ vector bundle $\EE$ on $Y$ and by a global
		section $\eta \in H^0(Y, \, S^3 \EE^{\vee} \otimes \bigwedge^2
		\EE)$, and conversely. 
	\end{theorem}
	
	The vector bundle $\EE$ is called the \emph{Tschirnhausen bundle} of
	the cover, and it satisfies
	\begin{equation}\label{eq_OOE}
		f_{*}\oo_X = \oo_Y \oplus \EE.
	\end{equation}
	
	By \cite[Theorem 1.5]{CE96} if $Y$ is smooth and the section $\eta\in H^0(Y, \, S^3 \EE^{\vee} \otimes \bigwedge^2	\EE)$ is generic, then $X$ is Gorenstein. 
	
	Let $D$ be a divisor such that $\mathcal{O}_Y(D)=\bigwedge^2\EE^{-2}$.

	\begin{proposition} \emph{\cite[Theorem 1.3]{Tan02}}\label{thm_Branch}
		Let  $f \colon X \to Y$ be a triple cover, assume that $Y$ be a normal variety.
		There exist two divisors $D'$ and $D''$ such that $D=2D'+D''$ and if $f$ is totally ramified then $D''=0$ and $D'$ is the branch divisor; otherwise $D$ is the branch divisor and $D'$ is the divisor over which $f$ is totally ramified. 
	\end{proposition}
	
	\begin{say}\label{say_diag}  We observe that there exists a divisor $D/2$ such that  $\mathcal{O}_Y(D/2)=\bigwedge^2\EE^{-1}$. By the previous proposition $D''=D-2D'$ is effective and it is two divisible (i.e. $D''=2(D/2-D')\in Pic(Y)$). Hence there exists a double cover of $Y$ branched on $D''$. This double cover is used to get the Galois closure of the triple cover whose Galois group is $\mathfrak{S}_3$ (see \cite{Tan02,CP}). We have the following digram:
		
		\begin{eqnarray}\label{diag split triple cover}\xymatrix{&&Z\ar[dd]_{\mathfrak{S}_3}\ar[dl]_{2:1}^{\alpha}\ar[dr]^{3:1}_{\beta_2}\\&X\ar[dr]_{3:1}^f&&W\ar[dl]^{2:1}_{\beta_1}\supset \beta_1^{-1}(D')\\
			&D'\cup D''\subset &Y&\supset D''}\end{eqnarray}
		
		We notice that the branch locus of $\beta_1$ is $D''$; $\beta_2$ is a Galois triple cover branched along $\beta_1^{-1}(D')$; the triple cover $f$ is totally branched on $D'$ and simply on $D''$. Finally it is worth to notice that this is a special case of dihedral cover studied in \cite{CP}.\\
	\end{say}
	
	If $Y$ is smooth, then $f$ is smooth over
	$Y - D$, in other words all the singularities of $X$ come from the
	singularities of the branch locus. More precisely, we have
	
	\begin{prop} \emph{\cite[Proposition 5.4]{Pa89}, \cite[Theorem 3.2]{Tan02}} \label{prop.sing.ram}
		Let $Y$ be a smooth variety. Let $y\in Y$, $f^{-1}(y)$ is a singular point of $X$ if and only if  $y\in \emph{Sing}(D)$ and one of the following conditions holds:
		\begin{itemize}
			\item[$(i)$] $f$ in not totally ramified over $y;$
			\item[$(ii)$] $f$ is totally ramified over $y$ and
			$\emph{mult}_y(D) \geq 3$.
		\end{itemize}
		So -- using the notation of Proposition \ref{thm_Branch} --  $X$ is smooth if and only if
		\begin{enumerate}
			\item $D'$ is smooth;
			\item $D''$ and $D'$ have no common points;
			\item $D''$ has only cusps as singular points where $f$ is totally ramified. 
		\end{enumerate}
	\end{prop}
We observe that $(1)$ is due to the multiplicity 2 of the divisor $D'$ in the branch divisor $D$ and that even if $D''$ is the divisor where $f$ is simply branched it could contain isolated points of total branch.
	
	\begin{prop} \emph{\cite[Theorem 4.1]{Tan02}} \label{prop.can.ris}
		Let $f \colon X \to Y$ be a triple cover of a smooth surface $Y$,
		with
		$X$ normal. Then there are a finite number of blow-ups
		$ \sigma \colon \widetilde{Y} \to Y$ of $Y$ and a commutative diagram
		\begin{equation} \label{dia.can}
			\begin{xy}
				\xymatrix{
					\widetilde{X}  \ar[d]_{\tilde{f}} \ar[rr]^{\tilde{\sigma}} & & X \ar[d]^{f} \\
					\widetilde{Y}   \ar[rr]^{\sigma} & & Y,  \\
				}
			\end{xy}
		\end{equation}
		where $\widetilde{X}$ is the normalization of $\widetilde{Y}
		\times_{Y} X$, such that $\tilde{f}$ is a triple cover with smooth
		branch locus. In particular, $\widetilde{X}$ is a resolution of
		the singularities of $X$ (in general this resolution is neither the minimal resolution nor it gives a minimal model of $X$).
	\end{prop}
	Following \cite[Paragraph 4]{Tan02}, we call $\widetilde{X}$ the \emph{canonical resolution} of $X$.

	In the case of
	smooth surfaces, one has the following formulae.
	\begin{prop} \emph{\cite[Propositions 4.7 and 10.3]{M85}} \label{prop.invariants}
		Let $f \colon X \rightarrow Y$ be
		a triple cover of smooth surfaces with Tschirnhausen bundle $\EE$.
		Then
		\begin{itemize}
			\item[$\boldsymbol{(i)}$] $h^i(X, \, \mathcal{O}_X)=h^i(Y, \,
			\mathcal{O}_Y)+h^i(Y, \, \EE)$ for all $i\geq 0;$
			\item[$\boldsymbol{(ii)}$] $\chi(\mathcal{O}_X) = \chi(\mathcal{O}_Y) + \chi(\EE) = 3\chi(\mathcal{O}_Y) + \frac{1}{2}c^2_1(\EE) - \frac{1}{2}c_1(\EE)K_Y - c_2(\EE)$.
			\item[$\boldsymbol{(iii)}$]
			$K^2_X=3K^2_Y-4c_1(\EE)K_Y+2c_1^2(\EE)-3c_2(\EE)$.
			\item[$\boldsymbol{(iv)}$] $e(X) = 3e(Y) - 2c_1(\EE)K_Y + 4c^2_1(\EE)-9c_2(\EE)$.
			
		\end{itemize}
	\end{prop}

	\begin{say}\label{say_Canonical}  Here we analyze shortly the relation between the canonical bundle of the covering surface and the base one. Let  $f\colon X \longrightarrow Y$ be a triple cover with Tschirnhausen bundle $\EE$. We assume that $Y$ is smooth and $X$ normal. Following \cite{CE96}, we observe that to each cover $f:X\ra Y$ of degree $d$ it is associated an exact sequence
		$$0\ra \mathcal{O}_Y\ra f_*\mathcal{O}_X\ra \EE^{\vee}\ra 0$$ whose dual sequence is $$0\ra \EE\ra f_*\omega_{X|Y}\ra\mathcal{O}_Y\ra 0,$$
		defining the relative dualizing sheaf $\omega_{X|Y}$. 
		
		If we assume that $X$ is Gorenstein, by \cite[Theorem 2.1]{CE96}, the ramification divisor $R$ satisfies $\mathcal{O}_X(R)=\omega_{X|Y}$ where $R$ is the set of the critical points of the map $f:X\ra Y$. 
		Being $X$ normal, following \cite{R87} we define the canonical divisor $K_X$ of $X$ as the Weil divisor whose restriction to the smooth locus is the canonical divisor. Since $X$ is assumed to be Gorenstein, $K_X$ is a Cartier divisor as well.
		
		The restriction of $f$ to $X_{0}$, the smooth locus of $X$, is a triple cover $f_{0}:X_{0}\ra f(X_{0})$. By Hurwitz ramification formula $$K_{X_{0}} = f_{0}^*(K_{f(X_{0})})+(\omega_{X|Y})_{|X_0}.$$
		Indeed by the definition of the canonical divisor we obtain the following equality
		\begin{equation}\label{canonicalformula1}
		K_X = f^*(K_Y)+R.
		\end{equation}

		In the sequel, we shall be particularly interested in the case when $Y$ is a K3 surface -- or in general when $Y$ is a surface with trivial canonical bundle--. Then the Equation \eqref{canonicalformula1} simplifies to 
			\begin{equation}\label{canonicalformula}
				K_X = f^*(K_Y)+R=R.
			\end{equation}
		In addition, by Equation \eqref{eq_OOE} and by duality for finite flat morphisms we obtain 
			\begin{equation*}
				f_*\oo_X(K_X)\cong \big(f_*\oo_X\big)^{\vee}\otimes \oo_Y(K_Y) \cong \oo_Y(K_Y) \oplus \big(\EE^{\vee} \otimes \oo_Y(K_Y)\big),
			\end{equation*}
			which yields at once 
			\begin{equation}\label{eq_h0k}
				h^0(\oo_X(K_X)) \geq h^0(\oo_Y(K_Y)).
			\end{equation}
	
		The formula \eqref{canonicalformula} is particularly useful to determine the Kodaira dimension of the covering surface $X$, indeed it holds:
	\end{say}
	
	\begin{prop}\label{prop_kodairaDim1}
		Let  $f\colon X \longrightarrow Y$ be a triple cover with $Y$ smooth and $X$ Gorenstein  and let $R$ be the ramification divisor. Write $|R|=|M|+F$ with $|M|$ the moving part and $F$ the fixed part of $R$. Suppose that $\oo(K_Y)\cong \oo_Y$, then the Kodaira dimension $\kappa(X)$ of $X$ is greater than or equal to $0$. Moreover,  it holds:
		\begin{description}
			\item[$\mathbf{\kappa(X)=0}$] if and only if $|M| = \emptyset$ and either $F$ is supported on rational curves or $F=\emptyset$. Moreover, $X$ cannot be neither an Enriques surface nor a bielliptic one. 
			
			\item[$\mathbf{\kappa(X)=1}$] if and only if $|M| \neq \emptyset$  and the general member of $|M|$ is supported on elliptic curves. 
			\item[$\mathbf{\kappa(X)=2}$] if and only if $|M| \neq \emptyset$  and the general member of $|M|$ is supported on curves of genus $g \geq 2$. 
		\end{description}
	\end{prop}
	\begin{proof} Suppose, first, that $f$ is \'etale of degree d, then $\kappa(X)=d\cdot \kappa(Y)=0$, hence we can assume that $f$ is ramified.

	Being $f$ a triple cover of a surface $Y$ with trivial canonical bundle by \eqref{eq_h0k} we have $h^0(\oo(K_X))\geq 1$ hence the Kodaira dimension of $X$ satisfies $\kappa(X) \geq 0$. Moreover, if $\kappa(X)=0$ then $X$ cannot be neither an Enriques surface nor a bielliptic one.

		By the Hurwitz  formula \eqref{canonicalformula} the canonical divisor is
		$K_X =R$. If the divisor $R$ is supported at least on a moving positive genus curve then  
		equation \eqref{canonicalformula} implies that $K_X$ is non-trivial, thus $\kappa(X)>0$. Moreover,  $X$ is a properly elliptic surface if and only its canonical bundle is supported only on elliptic curves (see \cite[Section III]{M89}), thus, again by \eqref{canonicalformula}, we have the second case of the Proposition. While, if the general member of $|R|$, is supported on a curve of genus $g\geq 2$ then  $X$ is of general type.
		
		Finally, if $R$ does not move in a linear system and it is supported only on rational curves, then the minimal model of $X$ must have trivial canonical bundle and so we get the last case of (1).
		
		Conversely, if $|M|=\emptyset$, $\kappa(X)$ can not be 1 or 2, hence $\kappa(X)=0$.
		If $|M| \neq \emptyset$  and the general member of $|M|$ is supported on curves of genus $g \geq 2$, $\kappa(X)$ can not be 0 or 1, hence it is 2.
		
		If  $|M| \neq \emptyset$  and the general member of $|M|$ is supported on elliptic curves, by adjunction $M\cdot M=0$ and so $X$ would have a genus 1 fibration, which is impossible for surfaces of general type. Since $|M| \neq \emptyset$, $\kappa(X)\neq 0$.
		
	\end{proof}
Similar results work in case $X$ is normal, but not necessarily Gorenstein. The main problem here is that the canonical divisor is not Cartier, so we have consider the canonical resolution of $X$.
\begin{prop}\label{prop_kodairaDim} The results of Proposition \ref{prop_kodairaDim1} hold even if $X$ is normal (not necessarily Gorenstein).\end{prop}
\begin{proof}
If $X$ is normal we consider the canonical resolution $\widetilde{X}$, as in Proposition \ref{prop.can.ris}, since $\kappa(X)=\kappa(\widetilde{X})$.
The map $\sigma:\widetilde{Y}\ra Y$ is a blow up and introduces an exceptional divisor $E$, which does not move in a linear system. The canonical bundle of $\widetilde{Y}$ is $K_{\widetilde{Y}}=K_Y+E=E$. We apply the Hurwitz ramification formula to the triple cover $\widetilde{f}\colon \widetilde{X}\ra\widetilde{Y}$:
$$K_{\widetilde{X}}=\widetilde{f}^*E+R_{\widetilde{f}},$$
where $R_{\widetilde{f}}$ is the ramification divisor of $\widetilde{f}$.

Let analyse both the summands: $\widetilde{f}^*E$  has negative self intersection and is the exceptional locus of $\widetilde{\sigma}$  for the commutativity of the diagram. Hence its movable part is trivial, otherwise there would be at least one curve in  $|\widetilde{f}^*E|$ which is not contracted by $\widetilde{\sigma}$, say $\widetilde{D}$. Denoted by $D$ the image of this curve, $\sigma^*(D)\sim \widetilde{D}+k\widetilde{f}^*E \sim (k+1)\widetilde{f}^*E$, for a non negative $k$. But this contradicts $\sigma^*(D)\widetilde{f}^*E=0$.

We observe that $R_{\widetilde{f}}$ are curves contained in the linear system of $$\widetilde{\sigma}^*R=\widetilde{\sigma}^*(M+F)=\widetilde{M}+\widetilde{F}+k\widetilde{f}^*E,$$
where $\widetilde{M}$ and $\widetilde{F}$ are the strict transforms of $M$ and $F$ respectively, using the notation of Proposition \ref{prop_kodairaDim1}. Since $\widetilde{f}^*E$ has no movable part, the movable part of $R_{\widetilde{f}}$ is contained in $\widetilde{M}$. Since $\sigma$ is a blow up $|M|$ is isomorphic to $|\widetilde{M}|$.

To conclude it suffices to apply the Proposition \ref{prop_kodairaDim1} to the cover $\widetilde{f}$.

\end{proof}
	
	\begin{say}\label{say_tripleGalois} The situation becomes a little bit easier if we consider only the Galois case. Indeed, by \cite[Theorem 2.1]{Tan02} a triple cover is Galois if and only if it is totally ramified over its branch locus. Moreover, in this instance the branch locus is exactly $D'=D/2=B+C$ \cite[Theorem 1.3 (2)]{Tan02}. In this case $X$ is smooth if and only if $D$ is smooth \cite[Theorem 3.2]{Tan02}. A Galois triple cover $f:X\rightarrow Y$ (with $Y$ smooth) is first of all a cyclic $\ZZ/3\ZZ$-cover, hence it can be treated as a cyclic cover. Moreover, it is determined by two curves $B$ and $C$ in $Y$ and by two divisors
		divisors $L,$ $M$ on $Y$ such that $B\in |2L-M|$ and $C\in |2M-L|$.
		As already remarked the branch locus of $f$ is $B+C$ and $3L\equiv 2B+C,$ $3M\equiv B+2C$. Thus the class of the branch locus is $B+C=L+M$.  The surface $X$
		is normal if and only if  $B+C$ is reduced. Otherwise it is possible to consider the normalization, which is associated to another triple cover as explain in \cite[Proposition 7.5]{M85}. The singularities of $X$ lie over the singularities of $D'=B+C$, see Proposition \ref{prop.sing.ram}.
		
		If $B+C$ is smooth, we have
		\begin{equation}\label{eq1}
			\chi(\mathcal O_X)=3\chi(\mathcal O_Y)+\frac{1}{2}\left(L^2+K_YL\right)+\frac{1}{2}\left(M^2+K_YM\right),
		\end{equation}
		\begin{equation}\label{eq2}
			K_X^2=3K_Y^2+4\left(L^2+K_YL\right)+4\left(M^2+K_YM\right)-4LM,
		\end{equation}
		\begin{equation}\label{eq3}
			q(X)=q(Y)+h^1(Y,\mathcal O_Y(K_Y+L))+h^1(Y,\mathcal O_Y(K_Y+M)),
		\end{equation}
		\begin{equation}\label{eq4}
			p_g(X)=p_g(Y)+h^0(Y,\mathcal O_Y(K_Y+L))+h^0(Y,\mathcal O_Y(K_Y+M)).
		\end{equation}
	\end{say}
	Of course, one can apply Theorem \ref{teo.miranda} to the Galois case. For this, let $\xi$ be a primitive cube root of unity, generating $\ZZ/3\ZZ$ and we obtain: 
	
	\begin{prop}\emph{\cite[Proposition 7.1]{M85}, \cite[Theorem 1.3]{Tan02}}\label{prop_MirandaGalois}
		If $f\colon X \longrightarrow Y$ is a Galois triple cover, then:
		\begin{itemize}
			\item[$\boldsymbol{(i)}$] The sheaf $f_*\oo_X$ splits into eigenspaces as $\oo_Y \oplus \mathcal{L}^{-1} \oplus \mathcal{M}^{-1}$ where $\oo_Y$, $\mathcal{L}^{-1}$ and $\mathcal{M}^{-1}$ are the eigenspaces for $1$, $\xi$ and $\xi^2$, respectively.
			\item[$\boldsymbol{(ii)}$] The Tschirnhausen bundle $\EE$ for $f$ is the sum of eigenspaces $\mathcal{L}^{-1} \oplus \mathcal{M}^{-1}$.
			\item[$\boldsymbol{(iii)}$] The branch locus of $f$ is the divisor $D'$ such that $\mathcal{O}(D')=\mathcal{L} \otimes \mathcal{M}$. 
		\end{itemize}
	\end{prop}
	\proof $(i)$ and $(ii)$ are contained in \cite[Proposition 7.1]{M85}. Recall that, by \cite[Theorem 2.1]{Tan02},  a triple coverer is Galois if and only if it is totally ramified over its branch locus. Moreover, in this case the branch locus is exactly $D'=D/2=B+C$ \cite[Theorem 1.3 (2)]{Tan02}.\endproof
	\begin{say}\label{say: Galois enigespance}
      Being Galois triple cover cyclic, one can compare the previous result with the standard theory of cyclic covers, see e.g. \cite[Chapter I.17]{BHPV}. Both the line bundles $\mathcal{M}$ of Proposition \ref{prop_MirandaGalois} and $\mathcal{L}^2$ correspond to the same eigenspace of $f_*\mathcal{O}_X$, the one relative to the eigenvalue $\xi^2$. Nevertheless the can differ in the Picard group by 3-torsion element and integer multiple of divisors supported on the codimension 1 subvarieties in the branch locus, see also \cite[Section 1.4]{Tan02}. The similar situation must occur for $\mathcal{M}^2$ and $\mathcal{L}$.
      
      Viceversa, given a  Tschirnhausen bundle $\EE$ and a section $\eta\in H^0(Y,S^3\EE^{\vee}\otimes \bigwedge^2\EE)$ it determines a triple cover and to see if it is Galois one has to check that $\mathcal{O}_Y\oplus\EE$ is a representation of the group $\Z/3\Z$. In particular $\EE$ must split according to the two characters $\xi$ and $\xi^2$. Therefore $\EE=\mathcal{L}^{-1}\oplus\mathcal{M}^{-1}$ where $\mathcal{L}^2$ differs from $\mathcal{M}$  by 3-torsion element and integer multiple of divisors supported on the codimension 1 subvarieties in the branch locus and 
	 $\mathcal{M}^2$ differs from $\mathcal{L}$ by 3-torsion element and integer multiple of divisors supported on the codimension 1 subvarieties in the branch locus. 
	 The choice of $\eta$ determines uniquely the cover.
\end{say}
	\begin{say}\label{say_GTD} Finally we call a \emph{Galois triple cover data} over $Y$  a pair of line bundles $\mathcal{L}$ and $\mathcal{M}$ on $Y$ and two sections
		\[
		b \in H^0(\mathcal{L}^2 \otimes \mathcal{M}^{-1}), \quad c \in H^0(\mathcal{L}^{-1} \otimes \mathcal{M}^2).
		\]
		To give a Galois triple cover data we will use alternatively the quadruple $(b,c,\mathcal{L},\mathcal{M})$  or $(B,C,\mathcal{L},\mathcal{M})$ or even $(B,C,L,M)$ with their clear meaning, i.e., the different cases of the letter give a different incarnation of the objects treated, once being a section, once a divisor and once a sheaf.
	\end{say}

	\begin{say} We assume that $X$ is normal and denote by $X'$ the \emph{minimal resolution} of the singularities of
		$X$. We observe that in general $X'$ coincides neither with the minimal model $X^{\circ}$ nor with the canonical resolution $\widetilde{X}$. One cannot apply neither the formulae of Proposition \ref{prop.invariants} nor the  \eqref{eq1},\eqref{eq2},\eqref{eq3} and \eqref{eq4} if $X$ is singular and in particular for the last ones if the branch locus is singular. We would like to define a class of singularities of $f:X\ra Y$  such that the formulae of Proposition \ref{prop.invariants} give the invariants of $X'$(instead of the one of $X$). 
		
		We follow \cite[Definition 1.5]{PP13}.
	\end{say}
	
	\begin{definition} \label{def.neg.sing}
		Let $f\colon X \rightarrow Y$ be a triple cover of a smooth
		algebraic surface $Y$, with Tschirnhausen bundle $\EE$. We say that
		$X$ has only \emph{negligible} $($or \emph{non essential}$)$
		singularities if the invariants of the minimal resolution $X'$ are
		given by the formulae in Proposition \emph{\ref{prop.invariants}}.
		We also call \emph{negligible} singularities the corresponding singularities of the branch locus.
	\end{definition}
	
	\begin{proposition}\emph{\cite[Examples 1.6 and 1.8]{PP13}}\label{prop: negligible} 
		Let  $f\colon X \rightarrow Y$ be a triple cover of a smooth
		algebraic surface $Y$, with Tschirnhausen bundle $\EE$. If the singularities of $X$ are only of type $\frac{1}{3}(1,1)$ and $\frac{1}{3}(1,2)$, then $X$ has only negligible singularities.  
	\end{proposition}

	\section{Triple covers of K3 surfaces: the Galois Case}\label{sec: the Galois case}
	
	From this section onwards, we shall always  consider as the base of a triple cover a K3 surface $S$. Unless otherwise stated $f:X\ra S$ is a triple cover such that $X$ is a normal connected surface.
	
	\subsection{Galois triple cover of K3 surface}
	Let $S$ be a K3 surface and $f:X\ra S$ be a Galois triple cover of $S$ with branch locus $\coprod_{i=1}^n D_i$ where $D_i$ are (possibly singular and reducible) curves and thus the $D_i$'s are the $n$ connected components of the branch locus. Requiring that $X$ is normal implies that the $D_i$'s are reduced. Up to reordering the components, we can always assume that $D_1^2\geq D_i^2$ for every $i=1,\ldots, n$.
	Since $B+C$ is the branch locus (with the same notation of Paragraph \ref{say_tripleGalois}), there exist curves $B_i$ and $C_i$ (not necessarily connected) such that $D_i=B_i+C_i$ and $B=\sum_i B_i$, $C=\sum_iC_i$.
	
	\begin{lemma}\label{lemma: intersections} Let $(B, C, L,M)$ be the Galois triple cover data of a Galois triple cover $X\ra S$. Let $B=\sum_{i=1}^nB_i$, $C=\sum_{i=1}^nC_i$ and $D_i=B_i+C_i$. Then:\begin{itemize}\item $B_iB_j=C_iC_j=B_iC_j=0$, if $i\neq j$; \item $B_i^2\equiv B_iC_i\equiv C_i^2\equiv_3 D_i^2$
	\end{itemize}\end{lemma}
	\proof Since $B_i$ (resp. $C_i$) are contained in a connected component of the branch locus and $B_j$ (resp. $C_j$) in another one, we have $B_iB_j=C_iC_j=B_iC_j=0$, if $i\neq j$. Moreover, 
	the last part of the statement follows directly by the conditions (not all independent): $B_iL=B_i(2B+C)/3=(2B_i^2+B_iC_i)/3\in \Z$. Analogously $C_iL=C_i(2B+C)/3\in \Z$, $L^2\in2\Z$,   $B_iM=B_i(B+2C)/3\in \Z$, $C_iM=C_i(B+2C)/3\in \Z$, $M^2\in2\Z$.\endproof
	
	\begin{proposition}\label{prop: smooth Galois triple cover, branch locus} If $f:X\ra S$ is a smooth Galois triple cover, then the branch locus is smooth. In particular:  
		\begin{itemize}
			\item all the $D_i's$ are irreducible smooth curves with positive genus; \item  if moreover $D_1^2>0$, then $n=1$ (i.e. the branch locus is $D_1$) and there exists a divisor $H\in Pic(S)$ such that $D=3H$;
			\item if at least one $D_i$ has genus 1, i.e. $D_i^2=0$, then all the $D_i$'s are smooth curves of genus 1. In particular  $\varphi_{|D_1|}:S\ra \mathbb{P}^1$ is an elliptic fibration, all the $D_i$'s are smooth fibers and $f:X\ra S$ is obtained by a base change of order 3 branched on the $n$ smooth fibers $D_i$'s.
		\end{itemize}
	\end{proposition}
	\proof The triple cover $f:X\ra S$ is smooth if and only if the branch locus is smooth (see Paragraph \ref{say_tripleGalois}), thus each $D_i$ is a smooth irreducible curve. In particular either $D_i=B_i$ or $D_i=C_i$. In both the cases one obtains $D_i^2\equiv_3 0$, by Lemma \ref{lemma: intersections}. If $D_i$ is a smooth irreducible rational curve, $D_i^2=-2\not\equiv_3 0$ which is not admissible. So for every $D_i$ we have $g(D_i)\geq 0$. If $D_1^2>0$, then, by Hodge index theorem $D_i^2<0$ for each $i>1$, which implies that $D_i$ are rational curves for each $i>1$, contradicting the first assertion. Therefore, if $D_1^2>0$, there are no other components in the branch locus. This implies that $B=D_1$ and $C=0$ (or vice versa). So $L=B/3\in Pic(S)$, i.e. there exists a divisor $H$ such that $3H\equiv D$.
	
	If $D_1^2=0$ and it is a smooth irreducible curve, then it is a genus 1 curve on $S$ and $\varphi_{|D_1|}:S\ra\mathbb{P}^1$ is a genus 1 fibration. In particular, any $D_j$ orthogonal to $D_i$ is contained in a fiber of $\varphi_{|D_1|}$ and thus it can be either a rational curve (component of a reducible fiber) or a genus 1 curve. Since there are no rational curves contained in the branch locus, we conclude that all the $D_i$'s are smooth fibers of the same fibration and that the triple cover $X\ra S$ is branched over smooth fibers of the fibration $S\ra \mathbb{P}^1$. This induces a genus 1 fibration $X\ra C$, where $C$ is a smooth curve, such that that there is a $3:1$ map $g:C\ra \mathbb{P}^1$ and $X\simeq S\times_g\mathbb{P}^1$. \endproof

	Let us now consider the case of a singular branch locus: in order to construct $X$ we first consider its  canonical desingularization $\widetilde{X}$. So we blow up $S$ in such a way that the strict transform of the branch locus becomes smooth and then we take the triple cover (see Proposition \ref{prop.can.ris}). The surface $X$ is a contraction of $\widetilde{X}$. This gives information both on the singularities of $X$ and on the construction of a smooth model of it. So now we perform a local analysis near the singularities of the branch locus. To simplify the treatment we work locally around a singular point $P$ and we assume it is the unique singularity of the branch locus. 
	
	\medskip
	
	\begin{say}{\it Singularities of the branch locus of type 1.}\label{say: singu type 1}
		
		Let $f:X\ra S$ be a triple cover. Let $W$ and $V$ two curves on $S$, meeting transversally in the point $P$. Let $W$ and $V$ be contained in the branch locus, and assume that $P$ is an ordinary double point of the branch. Let us assume that locally the equation of the branch locus of the triple cover near to $P=(0,0)$ is $xy$. We blow up $P$ obtaining $\beta_1:S_1\ra S$ which introduces an exceptional divisor $E_P$. We denote by $W_1$ and $V_1$ the strict transforms of $W$ and $V$ with respect to $\beta_1$. The divisor $E_P$ appears with multiplicity 2, so it is still contained in the branch locus of the triple cover and it intersects $W_1$ and $V_1$ in two points $R$ and $Q$. We further blow up $R$ and $Q$, obtaining the map $\beta_2:S_2\ra S_1$, the exceptional divisors $E_R$ and $E_Q$ and the strict transforms $W_2$, $V_2$ and $\widetilde{E_P}$ of $W_1$, $V_1$ and $E_P$ respectively. Denoted by $w=W^2$ and $v=V^2$, the intersection properties on $S_2$ are the following $E_R^2=E_Q^2=-1$, $W_2^2= w-2$, $V_2^2=v-2$, $\widetilde{E_P}^2=-3$, $E_RW_2=E_QV_2=E_R\widetilde{E_P}=E_R\widetilde{E_P}=1$, the other intersections are trivial. The triple cover of $f_2:X_2\ra S_2$, induced by the one of $S$, is branched on $W_2$, $V_2$ and $\widetilde{E_P}$ so the branch locus is smooth and $X_2$ is the canonical resolution of $f:X\ra S$. The self intersections of the inverse image of some curves on $S_2$ are the following: $$\begin{array} {ll}\left(f_2^{-1}(E_R)\right)^2=\left(f_2^{-1}(E_Q)\right)^2=-3,& \left(f_2^{-1}(W_2)\right)^2=(W^2-2)/3,\\ \left(f_2^{-1}(V_2)\right)^2=(V^2-2)/3,& \left(f_2^{-1}(\widetilde{E_P})\right)^2=-1\end{array}.$$

		To get $X$ we contract $f_2^{-1}(E_R)$, $f_2^{-1}(E_Q)$, $f_2^{-1}(\widetilde{E_p})$.  Since $(f_2^{-1}(\widetilde{E_P}))^2=-1$, the contraction $\gamma_2:X_2\ra X_1$ gives the minimal resolution $X'=X_1$ of $X$. Moreover, $X_2$ coincides with the canonical resolution $\widetilde{X}$. Then one contracts the curves $\gamma_2(f_2^{-1}(E_R))$ and $\gamma_2(f_2^{-1}(E_Q))$  which are $-2$ curves meeting in a point (i.e.\ a Hirzebruch-Jung string of type $\frac{1}{3}(1,2)$). The contraction $\gamma_1:X_1\ra X$ of these two curves produces the surface $X$ (triple cover of $S$). The image of these two curves under $\gamma_1$ is the point $f^{-1}(P)$, and thus $f^{-1}(P)$ is a singular point on $X$ of type $A_2$, i.e. of type $\frac{1}{3}(1,2)$.  We have the following diagram  (see also \verb|Figure 1|)
		$$\xymatrix{X\ar[d]_f&X_1\ar[l]^{\gamma_1}&X_2\ar[l]^{\gamma_2}\ar[d]_{f_2}\\ S&S_1\ar[l]^{\beta_1}&S_2\ar[l]^{\beta_2}}.$$

		\begin{center}
			\includegraphics[width=8cm]{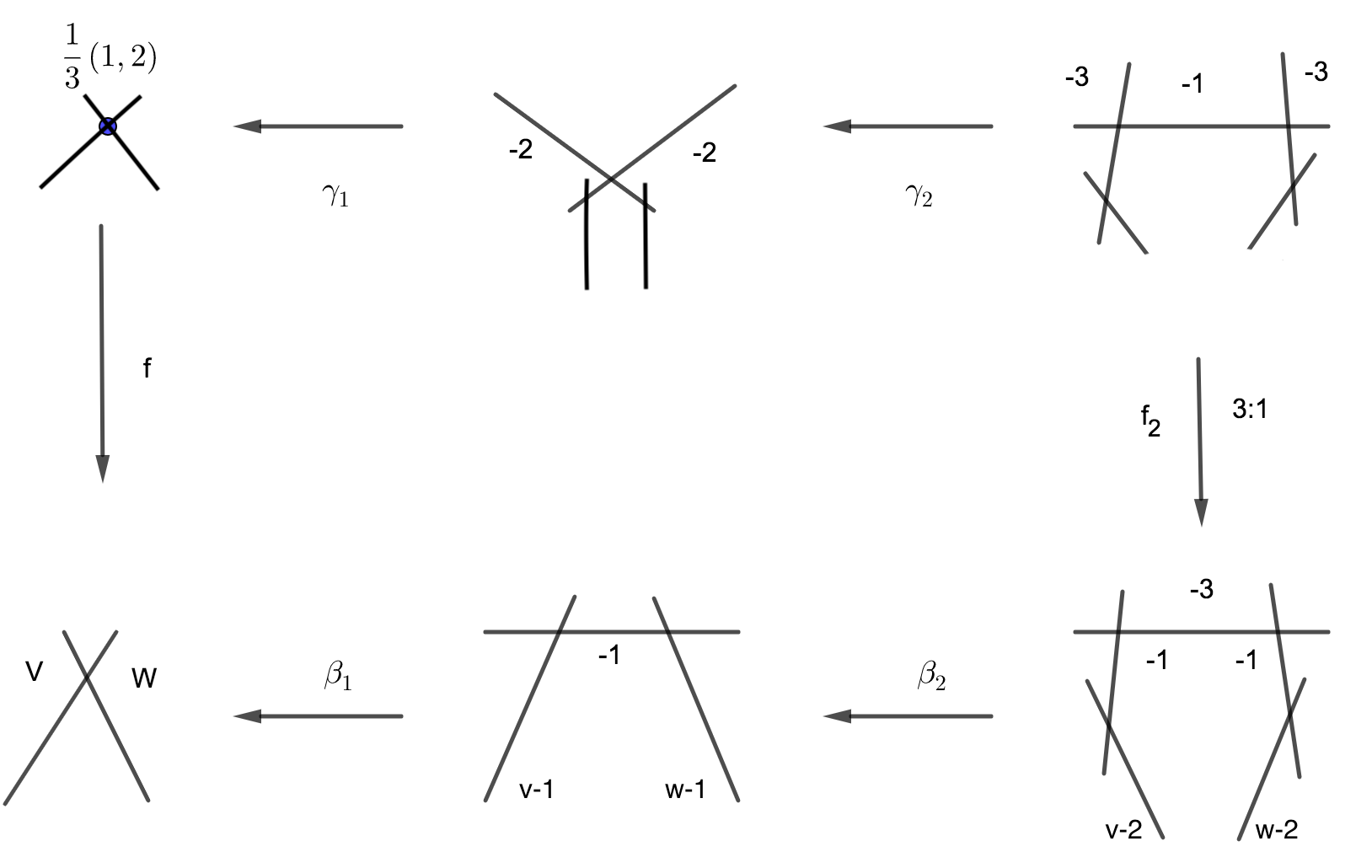}
			
			\verb|Figure 1|
		\end{center}
		
		\medskip
	\end{say}
	\begin{say}
		{\it Singularities of the branch locus of type 2.}\label{say: sing type 2}
		
		Let $f:X\ra S$ be a triple cover. Let $W$ and $V$ two curves on $S$, meeting transversally in the point $P$. Let $W$ and $V$ be contained in the branch locus, and assume that $P$ is an ordinary double point of the branch. Let us assume that locally the equation of the branch locus of the triple cover near to $P=(0,0)$ is $xy^2$. We blow up $P$ obtaining $\beta_1:S_1\ra S$ which introduces an exceptional divisor $E_P$. We denote by $W_1$ and $V_1$ the strict transforms of $W$ and $V$ with respect to $\beta_1$. 
		The divisor $E_P$ appears with multiplicity 3, so it is not contained in the branch locus of the triple cover.  The triple cover  $f_1:X_1\ra S_1$, induced by the one of $S$, is branched on $W_1\cup V_1$ and since the branch locus is smooth, $X_1$ is the canonical resolution $\widetilde{X}$ of $X$. The intersection properties on $S_1$ are the following $W_1^2=W^2-1$, $V^2=V_1^2-1$, $E_P^2=-1$, $W_1E_{P}=V_1E_P=1$, the other intersections are trivial. The self intersections of the inverse image of some curves on $S_1$ are the following: $$\left(f_1^{-1}(E_P)\right)^2=-3,\ \left(f_2^{-1}(W_1)\right)^2=(W^2-1)/3,\ \left(f_2^{-1}(V_1)\right)^2=(V^2-1)/3.$$ The contraction $\gamma_1:X_1\ra X$ of $E_P$ produces the surface $X$ (triple cover of $S$). The image of this curve under $\gamma_1$ is the point $f^{-1}(P)$, and thus $f^{-1}(P)$ is a singular point on $X$ of type $\frac{1}{3}(1,1)$.  In particular $X_1$ is also the minimal resolution $X'$ of $X$ which coincides in this case with $\widetilde{X}$. We have the following diagram
		$$\xymatrix{X\ar[d]_f&X_1\ar[l]^{\gamma_1}\ar[d]_{f_1}\\ S&S_1\ar[l]^{\beta_1}.}$$
	\end{say}
	
	We can summarize the above discussion with the following proposition.
	\begin{proposition}
		Let $f:X\ra S$ be a Galois triple cover. Let $P$ be a singular point of the branch locus.
		\begin{itemize} \item If the local equation of the branch locus near $P$ is $xy$, then $f^{-1}(P)$ is a singularity of $X$ type $\frac{1}{3}(1,2)$; 
			\item If the local equation of the branch locus near $P$ is $xy^2$, then $f^{-1}(P)$ is a singularity of $X$ type $\frac{1}{3}(1,1)$. 
		\end{itemize}
	
	\end{proposition}
	In both the previous cases $f^{-1}(P)$ is a negligible singularity by Proposition \ref{prop: negligible}.
	
	\begin{remark}{\rm 
			Let $f:X\ra S$ be a triple cover with Galois cover data $(B,C,L,M)$. 
			
			Let $V$ and $W$ two irreducible components of the branch locus, meeting in a point $P$. If $P$ is a singularity of type 1, then $V$ and $W$ are both components either of $B$ or of $C$.
			If $P$ is a point of type 2 then, w.l.o.g., $V$ is a component of $B$ and $W$ is a component of $C$.
			
			We observe the ordinary nodes of a curve (in the branch locus) are singularities of type 1.
		}  	
	\end{remark}
	
	By considering the contraction $\gamma_1$ described in \ref{say: singu type 1} and \ref{say: sing type 2} one obtains the following.
	\begin{corollary}
		Let $X\ra S$ be a Galois triple cover such that the branch locus contains $c$ singularities of type $1$ and $b$ singularities of type 2 and no other singular points. The Picard number of  the minimal resolution $X'$ of $X$ is $\rho(X')=\rho(X)+2c+b$, the ones of the canonical resolution  $\widetilde{X}$ is $\rho(\widetilde{X})= \rho(X)+3c+b$ (see also Figure 1).
	\end{corollary}

	In the following lemma and examples we will show that singularities both of type 1 and of type 2 appear as branch locus of triple cover of K3 surfaces.
	\begin{lemma}\label{lemma: type of sing} Let $N_1$ and $N_2$ be two rational curves contained in the branch locus of a Galois triple cover $X\ra S$ and such that any other curves contained in the branch locus is disjoint from them. Let us assume that $N_1N_2=k$. 
		Then:\\ $k\not \equiv_3 0$;
		\\ if $k\equiv_3 1$ the points of intersection of $N_1$ and $N_2$ are of type 2;
		\\ if $k\equiv_3 2$ the points of intersection of $N_1$ and $N_2$ are of type 1.\end{lemma} 
	\proof Since $N_1$ and $N_2$ have trivial intersection with any other components of the branch locus, the condition $LN_i\in \Z$, restricts to the condition $\mbox{$N_i(\alpha N_1+\beta N_2)/3\in \Z$}$, where $\alpha\in\{1,2\}$ and $\beta\in\{1,2\}$ and $(\alpha N_1+\beta N_2)/3$ is the summand of $L$ in which $N_1$ and $N_2$ appear with non trivial coefficient. Up to replace possibly $L$ with $M$, we can assume that $\alpha=1$. So the condition is now $N_i(N_1+\beta N_2)/3\in \Z$ which implies $-2+\beta k\equiv_3 0$ and $k-2\beta\equiv_3 0$. These conditions imply that $k\not\equiv_3 0$, $k\equiv_3 1$ if an only if $\beta=2$ and $k\equiv_3 2$ if and only if $\beta=1$.\endproof
	\begin{example} Let us consider an elliptic fibration $\mathcal{E}:S\ra\mathbb{P}^1$ with 2 fibers of type $I_2$ over the point $p_1$, $p_2$. The base change of order 3 $f:\mathbb{P}^1\ra\mathbb{P}^1$ branched over the points $p_1$ and $p_2$ induces a triple cover of $S$ whose branch locus contains 4 singularities of type 1. Indeed the branch curves are rational curves meeting in 2 points and then the singularities are of type 1 by Lemma \ref{lemma: type of sing}.
	\end{example}
	The following example contains singularities of type 2 and moreover shows that the minimal resolution $X'$ of $X$ constructed above is not necessarily a minimal model: 
	
	\begin{example} Let $f:S\ra\mathbb{P}^1$ be an elliptic fibration with 6 fibers of type $I_3$ and a 3-torsion section. Generically $\rho(S)=14$, see \cite[Section 4.1]{GS07}. Let $\Theta_i^{(j)}$, $i=0,1,2$ and $j=1,\ldots, 6$ be the irreducible components of the reducible fibers. There exists a $3:1$ cover $X\ra S$ branched over $\Theta_1^{(j)}$, $\Theta_2^{(j)}$ for $j=1,\ldots, 6$. The minimal model $X^{\circ}$ is a K3 surface, see \cite[Proposition 4.1]{G13}. In particular this triple cover is branched over 6 configurations of type $A_2$ of rational curves and the branch data are $B=\sum_{j=1}^6\Theta_1^{(j)}$, $C=\sum_{j=1}^6\Theta_2^{(j)}$, $L=(2B+C)/3$ and $M=(B+2C)/3$. 
		
		The surface $X^{\circ}$ has Picard number equal to 14, see \cite[Proposition 4.1]{GS07}, but the minimal resolution of $X'$ has Picard number $\rho(X')=\rho(S)+6$. In particular this implies that $X^{\circ}\neq X'$, i.e. the minimal resolution of $X$ is not minimal.
	\end{example}
	
	The minimal resolution of $X$ is not a minimal model if there is configuration of type $A_2$ in the branch locus of $X\ra S$, indeed in this case one has to contracts 2 $(-1)$-curves for each of these configurations, as show in  \verb|Figure 2|.

	\begin{center}
		\includegraphics[width=8cm]{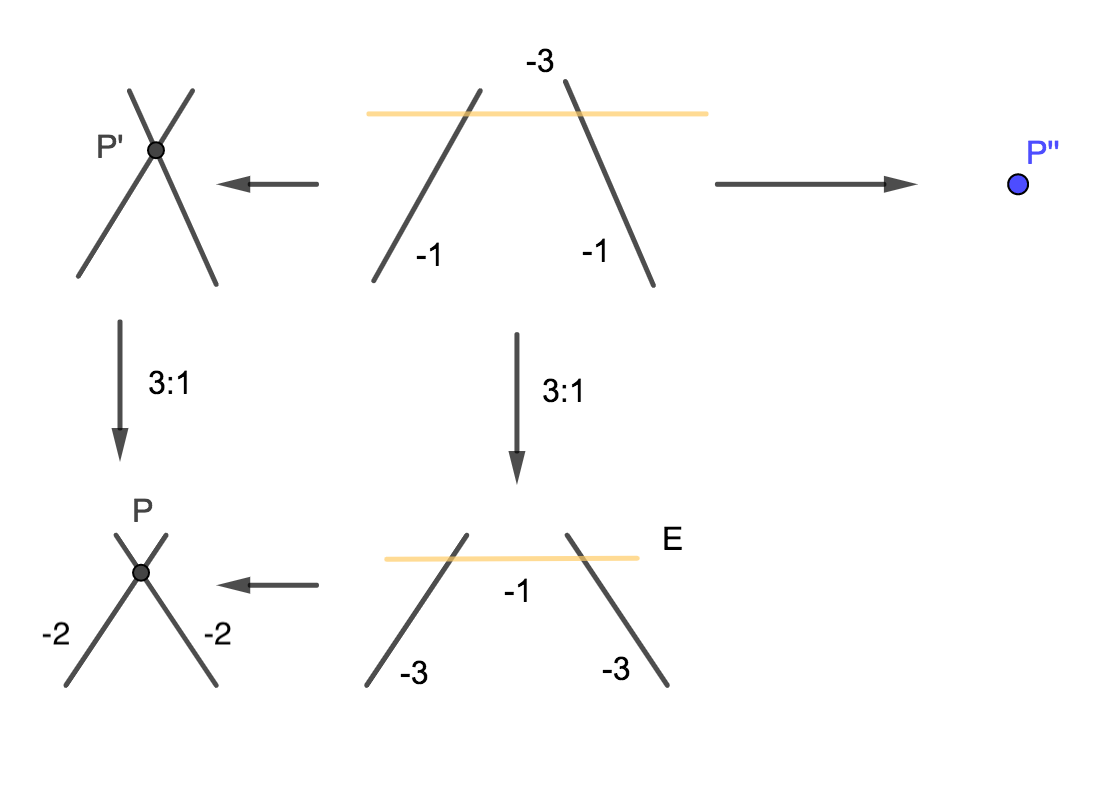}
		
		\verb|Figure 2|
	\end{center}

	\begin{rem}\label{rem: contraction of curves singularities of type 2}{\rm Let $f:X\ra S$ be a triple cover, $X'$ the minimal resolution of $X$ and $X^{\circ}$ the smooth minimal model of $X$. By the above construction one obtains that if a singularity of type 2 lie on exactly one rational curve which contains no other singularities of the branch locus, then its inverse image in $X'$ is contractible and so $K_{X^{\circ}}^2\geq K_{X'}^2+1$. If the singularity of type 2 lie on the intersection of two rational curves which do not contain other singularities of the branch locus, then in $X'$ there are three contractible curves (this is the case of Figure 2) and so $K_{X^{\circ}}^2\geq K_{X'}^2+3$. We refer to the last situation as $A_2$-configuration in the branch locus.}\end{rem}

	In Theorem \ref{say: other singularities in the branch locus} we will analyze other type of singularities which can appear in the branch locus of a triple cover.

	\medskip

	\begin{proposition}\label{prop: rito numbers}
		Let $X$ be a normal variety which is a Galois triple cover of $S$, whose branch locus has $n$ connected components, i.e., the $n$ connected reducible curves $D_1,\ldots, D_n$. Then 
		
		\begin{enumerate}
			\item Each curve $D_i$ is reduced;
			\item At most one $D_i$ is a curve with positive self intersection.

			\item  If the singularities of $X$ are negligible, denoted by $(B,C,L,M)$ the Galois triple cover data, the invariants of the minimal resolution $X'$ are 
			\[
			\begin{split} \chi(\mathcal{O}_{X'})= & \, 6+\frac{1}{2}(L^2+M^2),\\
				K^2_{X'}= & \, 2L^2+2M^2+LM,\\ 
				e(X')=& \, 72+4(L^2+M^2)-LM,\ h^{1,0}(X')=h^1(L)+h^1(M).
			\end{split}
			\]
		\end{enumerate}
	\end{proposition}
	\proof Condition $(1)$ is needed to obtain a normal cover $X\ra S$, see paragraph \ref{say_tripleGalois}. 
	
	Condition $(2)$ follows by the Hodge Index Theorem. 
	
	Since the singularities are negligible, the numerical invariants for $X'$ can be calculated using the Tschirnhausen bundle $\EE=\mathcal{L}^{-1} \oplus \mathcal{M}^{-1}=\mathcal{O}_S(-L) \oplus\mathcal{O}_S(-M)$. It has the following Chern classes
	\[ c_1(\EE)=-L-M, \quad c_2(\EE)=L\cdot M.
	\]
	To have a connected covering we have:
	\[ 0=h^0(\EE)=h^0(-L)+h^0(-M) \, \Rightarrow \, h^0(-L)=h^0(-M)=0.
	\]
	In addition, by Hirzebruch--Riemann--Roch theorem we have
	\begin{equation}\label{eq_chiT}\begin{split}
			\chi(\EE)=h^2(\EE)-h^1(\EE)=h^0(L)+h^0(M)-h^1(\EE)&=\\ =\int_S \textrm{Td}(S) \cdot \textrm{ch}(\EE) = 
			\int_S(1,0,2)\cdot\left(2,-L-M,\frac{L^2+M^2}{2}\right)&=\frac{L^2+M^2}{2}+4.
		\end{split}
	\end{equation} 
	So by Proposition \ref{prop.invariants} (ii) we have:
	$\chi(\mathcal{O}_{X'})=6+\frac{1}{2}(L^2+M^2)$,
	by Proposition \ref{prop.invariants} (iii) we have:
	$K^2_{X'}=2L^2+2M^2+LM
	$ and finally, by Proposition \ref{prop.invariants} (iv) we have:
	$e(X')=72+4(L^2+M^2)-LM.$
	
	By equation \eqref{eq_chiT} and the genus formula:
	\begin{equation}
		\begin{split}
			h^1(\EE)=h^0(L)+h^0(M)-\chi(\EE)=\left(\chi(L)+h^1(L)\right)+\left(\chi(M)+h^1(M)\right)-\chi(\EE)= \\
			= \left(\frac{L^2}{2}+2+h^1(L)\right)+\left(\frac{M^2}{2}+2+h^1(M)\right)-\frac{L^2+M^2}{2}-4=h^1(L)+h^1(M),
		\end{split}
	\end{equation} which allows to compute $h^{1,0}(X')$, by Proposition \ref{prop.invariants} (i).
	
	\endproof
	\begin{rem}
{\rm If the branch locus of the cover $f:X\ra S$ is smooth, then $BC=0$ and so the formulae in Proposition \ref{prop: rito numbers} and \eqref{eq2} which compute the self intersection of the canonical divisor of $X$ coincides and give $K_X^2=12(B^2+C^2)$.}	\end{rem}

	\begin{lemma}\label{lemma: -2 curves}
		Let us consider a connected reducible reduced curve, $C$, on a K3 surface $S$, whose irreducible components are rational curves (i.e. negative curves). If $C^2\leq 0$, then $C^2=-2$.
		
		In particular, a reducible reduced negative curve $D$ such that $D^2<-2$ has more than one connected component.
	\end{lemma}
	\proof Let us call $C_i$ the irreducible components of $C$. Then $C=\sum_{i=1}^n C_i$ and 
	\[ (C)^2=-2n+2\sum_{i=1}^n\sum_{j=i+1}^{n}C_iC_j.
	\] 
	$C$ is connected and each curve $C_i$ intersects at least one other curve, one has $\sum_{i=1}^n\sum_{j=i+1}^{n}C_iC_j\geq n-1$. Hence $C^2\geq -2n+2(n-1)=-2$. If $C^2<0$, then $C^2=-2$.\endproof

	\begin{lemma}
		Let $S$ be a K3 surface, and $\Lambda$ a sublattice of $NS(S)$. If $\Lambda$ is generated by irreducible curves and is degenerate, then all the curves represented by classes in $\Lambda$ are either fibers or components of fibers of a genus 1 fibration on $S$.
	\end{lemma}
\proof Since $\Lambda$ is degenerate, there exists $v\in\Lambda$ such that $v^2=0$ and $vw=0$ for every $w\in \Lambda$. By Riemann--Roch Theorem either $v$ or $-v$ is effective, so we assume that $v$ is effective and hence $\varphi_{|v|}$ is a fibration whose smooth fibers are genus 1 curve, cf. e.g. \cite[Proof of Lemma 2.1]{Ko}. Since any other class $w$ in $\Lambda$ has a trivial intersection with $v$, if a curve is represented by $w$, it is contracted by $\varphi_{|v|}$ and hence it is contained in a fiber of the fibration $\varphi_{|v|}$. \endproof  	

Let $\Lambda_{D_i}$ be the lattice generated by the irreducible components of an effective divisor $D_i$.
\begin{theorem}\label{theorem: three cases giusto}
	Let $f\colon X \ra S$ be a normal Galois triple cover of a K3 surface, whose branch locus has $n \geq 1$ connected components, i.e., the $n$ connected possibly reducible curves $D_1,\ldots, D_n$. 
	
	\begin{itemize}
		\item[1)] If $\kappa(X)=0$, then all the lattices $\Lambda_{D_i}$ are $A_2(-1)$ (in particular negative definite).
	    \item[2)] If $\kappa(X)=2$, then there exists a lattice $\Lambda_{D_i}$ such that $\mbox{$\rm{sgn}(\Lambda_{D_i})=(1,\rk(\Lambda_{D_i})-1)$}$ and all the other lattices $\Lambda_{D_j}$ are negative definite.
		\item[3)] If $\kappa(X)=1$, then there are no lattices $\Lambda_{D_i}$ which are indefinite and there exists at least on $\Lambda_{D_i}$ which is degenerate.
	\end{itemize}
\end{theorem}
\proof $1)$ If $\kappa(X)=0$, each component of $D_i$ is a negative curve, and thus by adjunction a $(-2)$-curve, since the intersection with the canonical bundle is trivial for every curve on $X$ (if $X$ is a K3 surface see also Lemma \ref{lemma: -2 curves}). Therefore, the lattice $\Lambda_{D_i}$ is a negative definite root lattice, and thus $\Lambda_{D_i}(-1)$ is an A-D-E lattice, \cite[Lemma 2.12, Chapter I]{BHPV}. By \cite[Chapter I, Section 17]{BHPV} the existence of a triple cover branched in $D_i$ implies that a linear combination of the components of the $D_i$'s with non integer coefficients in $\frac{1}{3}\Z$ is contained in the Picard group of $S$. So the discriminant group of $\Lambda_{D_i}$ contains $\Z/3\Z$, which implies that either $\Lambda_{D_i}\simeq A_{3k-1}(-1)$ with $k>1$ or $\Lambda_{D_i}\simeq E_6(-1)$. In both the cases $\Lambda_{D_i}^{\vee}/\Lambda_{D_i}$ is a cyclic group, generated by a class $l$ (supported on the components of $D_i$). The generators of the subgroups $\Z/3\Z$ in  $\Lambda_{D_i}^{\vee}/\Lambda_{D_i}$ are multiple of $l$ and they are always supported on a disjoint $A_2$-configurations of curves. So each $D_i$ is the sum $B_i+C_i$ where $B_i$ and $C_i$ are two rational curves meeting in a point, i.e. $\langle B_i,C_i\rangle\simeq A_2(-1)$.

$2)$ If $\kappa(X)=2$, then there exists a curve $C\in |R|$, where $R$ is the ramification of $f$ such that $g(C)>1$ (see Proposition \ref{prop_kodairaDim}) and $C^2\geq 0$. By the projection formula $\left(f(C)\right)^2\geq 0$ and so $g(f(C))\geq 1$. There exists a (possibly reducible) branch curve $\Gamma$ linearly equivalent to $f(C)$ and contained in $D_i$. If $g(f(C))>1$, then $\Gamma^2>0$ and hence $\Lambda_{D_i}$ satisfies the hypothesis, by Hodge index Theorem. 

If $g(f(C))=1$, $f(C)$ defines a genus 1 fibration $\mathcal{E}$ on $S$. Moreover $f(C)$ intersects the branch locus, i.e. $f(C)\cap(\cup_iD_i)\neq \emptyset$, otherwise $g(C)=1$. This implies that $\Gamma$ in the union of a fiber of $\mathcal{E}$ and some horizontal curves. In particular $\Gamma$, and hence $D_i$, contains the class $F$ of the fiber of the fibration $\mathcal{E}$ and at lest an horizontal curve $\Delta$. The class $\alpha F+\delta\Delta$ is contained in $\Lambda_{D_i}$ and it has a positive self intersection for $\alpha$ sufficiently big. By Hodge index theorem the signature of all the lattice $\Lambda_{D_i}$ are the required ones.

$3)$ If $\kappa(X)=1$, then there exists a curve $C\in |R|$, where $R$ is the ramification of $f$ such that $g(C)=1$ and $C$ defines a genus 1 fibration on $X$. Hence $f(C)$ defines a genus 1 fibration on $S$. All the irreducible components of $|f(C)|$ are genus 1 curve and $f(C)^2=0$. As above there exists a (possibly reducible) branch curve $\Gamma$ linearly equivalent to $f(C)$ and contained in $D_i$. Then $\Gamma\in |f(C)|$ and it is a fiber (irreducible or not), so it coincides with $D_i$ for a certain $i$ and $\Lambda_{D_i}$ is degenerate. All the others $D_j$'s are contained in (or coincides with) fibers, so each $\Lambda_{D_j}$ can not contain a class with positive self intersection. \endproof

\begin{corollary}\label{cor: giusto}
Let $f\colon X \ra S$ be a normal Galois triple cover of a K3 surface, whose branch locus has $n \geq 1$ connected components, i.e., the $n$ connected reducible and reduced curves $D_1,\ldots, D_n$. 
\begin{itemize}
	\item[1)] 
	If all the $\Lambda_{D_i}$ are negative definite, then $\kappa(X)=0$, $n=6$ or $n=9$ and the minimal model of $X$ is either a K3 or an Abelian surface.
	\item[2)]
	If ${\rm sgn}(\Lambda_{D_1})=(1,\rk(\Lambda_{D_1})-1)$, then $\kappa(X)=2$, $\Lambda_{D_j}\simeq A_2(-1)$ for all $j\neq 1$ and $n\leq 10$.
	\item[3)] If $\Lambda_{D_1}$ is degenerate, then $\Lambda_{D_j}$ for $j\neq 1$ is either degenerate or negative definite and $\kappa(X)=1$.	
\end{itemize}

\end{corollary}
\proof 
The assumption that the curve $D_i$'s are reduced guarantees that $X$ is normal.
$(1)$ If all the $\Lambda_{D_i}$ are negative definite, we already showed in Theorem \ref{theorem: three cases giusto} that $\Lambda_{D_i}\simeq A_2(-1)$.
The sets of $A_2$-configurations which are 3-divisible (and thus are in the branch locus of a triple cover) necessarily contain either 6 or 9 $A_2$-configurations \cite[Lemma 1]{B98}. In the first case the minimal model of the triple cover is a  K3 surface, in the latter it is an Abelian surface and in particular $\kappa(X)=0$.

$(2)$ If sgn$(\Lambda_{D_1})=(1,\rk(\Lambda_{D_1})-1)$, by Hodge index theorem, the lattices $\Lambda_{D_j}$ with $j\neq 1$ are negative definite. As in proof of $(1)$ of Theorem \ref{theorem: three cases giusto}, this implies that $\Lambda_{D_j}$ coincides with $A_2(-1)$. The maximal number of disjoint $A_2$-configurations of rational curves on a K3 surface is 9, hence $n\leq 10$.
To show that $\kappa(X)=2$, we observe that if $\kappa(X)=0$, then by Theorem \ref{theorem: three cases giusto} all $\Lambda_{D_i}$ should be negative definite,  which contradicts the hypothesis on $\Lambda_{D_1}$. If $\kappa(X)=1$ then by Theorem \ref{theorem: three cases giusto} at least one $\Lambda_{D_i}$ should be degenerate, which is again impossible. 

$(3)$ If $\Lambda_{D_1}$ is degenerate, then $D_1$ consists of a fiber of an elliptic fibration. So $D_j$, $j\neq 1$ are contained in fibers and then $\Lambda_{D_j}$ cannot contain a positive class. In particular $\kappa(X)$ can not be 2, because there is no an indefinite lattice $\Lambda_{D_i}$ and can not be 0, because not all the lattices $\Lambda_{D_i}$ are negative definite.
\endproof

We now consider the case in which one connected component $D_i$ of the branch locus is irreducible. The following corollary shows that if its self intersection is non negative, we can easily determine the Kodaira dimension of $X$. 	
\begin{corollary}\label{cor: giusto irreducible}
Let $D_i$ be an irreducible reduced curve in the branch locus of a Galois triple cover $f:X\ra S$.
If $D_i^2>0$ then $\kappa(X)=2$.
If $D_i^2=0$ then $\kappa(X)=1$.

If $D_i^2\geq 0$, then the minimal model $X^{\circ}$ of $X$ is obtained by contracting on $X'$ three curves for each $A_2$-configuration in the branch locus.
\end{corollary}	
\proof
The first part of the statement follows directly by Corollary \ref{cor: giusto}.

Let $D_1$ be a smooth irreducible curve, then the singularities of the branch locus are of type 2 and they are the singular point of each the $A_2$-configurations in the branch locus.

In this case $X^{\circ}$ is obtained by contracting on $X'$ three curves for each $A_2$-configuration in the branch locus. Indeed by Remark \ref{rem: contraction of curves singularities of type 2} one has to contract at least 3 curves for each $A_2$-configuration and we call the surface obtained by these contraction $X_{m}$. It remains to prove that $X_{m}$ and $X^{\circ}$ coincide, and so that there are no other possible contractions to a smooth surface. 

On $X_{m}$ there is an automorphism $\sigma_{m}$ of order 3, induced by the Galois $\Z/3\Z$-cover automorphism $\sigma$ on $X$. The surface $S_{m}:=X_{m}/\sigma_{m}$ is the singular surface obtained by $S$ contracting all the $A_2$-configuration in the branch locus. If there were a $(-1)$ curve $E$ on $X_{m}$ then there are two cases: either $E$ is disjoint from the ramification locus of $f_{m}:X_{m}\ra S_{m}$, or $E$ meets it (not being contained).

If $E$ is disjoint from the ramification locus, then there exists a $(-1)$-curve on $X$ which is mapped to $E$ (because we blow up and down points away from $E$). Hence $\sigma(E)\cap E$ is empty therefore $\sigma(E)$ and $\sigma^2(E)$ are other two $(-1)$-curves on $X$. So $f(E)=f(\sigma(E))=f(\sigma^2(E))\subset S$ is a $(-1)$-curve. This is absurd because $S$ is a K3 surface.

Hence $E$ meets the ramification locus $R_{m}$ and $\sigma_{m}(E)=E$. Otherwise there should be different $(-1)$-curves ($E$ and $\sigma_{m}(E)$) meeting in the point $E\cap R_{m}$ and this is impossible on a surface with non negative Kodaira dimension. Moreover $\sigma$ is an automorphism of order 3 of the rational curve $E$, and then $E\cap R_{m}$ consists of two points

Let $\beta:X_{m}\ra X^{\circ}$ the contraction of $E$ and $\sigma^{\circ}$ the automorphism induced by $\sigma_{m}$ on $X^{\circ}$. This induces the contraction $S_{m}\ra S^{\circ}$ of the curve $f_{m}(E)$. Since $E\cap R_{m}$ consist of two points, the contraction of $f_{m}(E)$ identifies two singular points on $S_{m}$ and introduces a singularity on the image of $D_1$. By construction the smooth $X^{\circ}$ is a triple cover of $S^{\circ}$ and thus the singularities of $S^{\circ}$ cannot be worst than $(\mathbb{C}^2,0)/\Z/3\Z$ so two singularities of $S_{m}$ can not be identified. Moreover, there cannot be a singularity on the branch curve image of $D_1$, $S^{\circ}$ has to coincides with $S_{m}$. We conclude that $X_{m}$ coincides with $X^{\circ}$.\endproof

\section{Examples of Galois covers of K3 surfaces}\label{sec: examples Galois cover of K3 surfaces}

By Theorem \ref{theorem: three cases giusto} and Corollary \ref{cor: giusto}, $k(X)=0$ if and only if for all the components $D_i$ of the branch locus, $\Lambda_{D_i}$	is negative. In this case the minimal model of $X$ is either a K3 surface or an Abelian surface and these cases are well known (see e.g. \cite{B98}).

We now provide examples for the other cases. 
First we consider surfaces of general type: in Proposition \ref{cor: example general type, irreducible D1} and in Proposition \ref{prop: case (2) theorem- possibilities}
 we compute the invariants of $X^{\circ}$ if it is of general type and we make specific assumption on the component $D_1$ of the branch locus such that $\Lambda_{D_1}$ is indefinite;
in Corollary \ref{cor: pg=2 gen.type} we provide an example of surface $X^{\circ}$ with $p_g=2$; in Theorems \ref{theo: cover S16}, \ref{theo: cover S15} examples of $X^{\circ}$ with $q\neq 0$.
Then we consider the case $\kappa(X^{\circ})=1$ and we classify the invariants of $X^{\circ}$ in Proposition \ref{prop: examples case 3 of proposition}.
 
 \subsection{The covering surface $X$ is of general type}
	
	The case $(2)$ of Theorem \ref{theorem: three cases giusto} is the most general one, but under some assumptions we are able to give a more detailed description of these covers.
	
	We first assume to be in the hypothesis of Corollary \ref{cor: giusto irreducible}.

\begin{prop}\label{cor: example general type, irreducible D1} Let $D_1$ be a connected component of the branch locus of a Galois triple cover $f:X\ra S$, which is also irreducible reduced and of positive genus. Then $D_1^2=6d$, for an integer $d>0$; $d\equiv_3 n-1$ and denoted by $k$ the integer such that $d=n-1+3k$ one has 
	$k\geq -2$. 
	
	If $D_1$ is moreover smooth then \begin{equation}\label{eq: inv Xm special case}\chi(X^{\circ})=5+n+5k,\ K_{X^{\circ}}^2=8n-8+24k, \ e(X^{\circ})=68+4n+36k\end{equation}
\end{prop}
\proof  By Corollary \ref{cor: giusto} the components $D_j$ with $j>1$ consists of two rational curves meeting in a point. We denote by $A_1^{j}$ and $A_2^{j}$ the two components of $D_j$. Since $D_1$ is irreducible, it is a component of $B$ (or equivalently of $C$). Then the data of the triple cover are $$B=D_1+\sum_{j=1}^{n-1}A_1^{j},\ \ C=\sum_{j=1}^{n-1}A_2^{j}$$ $$L=\frac{D_1+\sum_{j=1}^{n-1}\left(A_1^{j}+2A_2^{j}\right)}{3},\ \ M=\frac{2D_1+\sum_{j=1}^{n-1}\left(2A_1^{j}+A_2^{j}\right)}{3}.$$
By $LD_1\in \Z$ and $L^2\in 2\Z$ it follows that $D_1^2\equiv_3 0$ and $d\equiv_3 n-1$. Since $n\leq 10$ and $d>0$, $9+3k>0$, so $k\geq -2$. The formulae \eqref{eq: inv Xm special case} follows by Proposition \ref{prop: rito numbers} since $L^2=2k$, $M^2=2n+8k-2$ and $LM=n-1+4k$ and $X^{\circ}$ is obtained by $X$ contracting $3(n-1)$ curves. We conclude by Corollary \ref{cor: giusto irreducible}.\endproof

\begin{say}
	In the situation of the previous corollary, if $L^2\geq -2$, there exists a member of the linear system $|D_1|$ which splits into the union of a curve $G$ and the curves $A_i^{j}$, where $G\simeq \frac{D_1-\sum_{j=1}^{n-1}\left(A_1^{j}+2A_2^{j}\right)}{3}$ and hence $D_1\simeq 3G+\sum_{j=1}^{n-1}\left(A_1^{j}+2A_2^{j}\right)$. We observe that $GA_2^{(j)}=1$ and $GA_1^{(j)}=0$. Since $G^2=2k$, if $G$ is an irreducible and smooth curve, then $g(G)=k+1$. In particular $G$ is rational if $k=-1$. A limit case is the one with $k=-1$ and $n=3$. It is not an example of case $(1)$ of the Theorem \ref{theorem: three cases giusto}, because $D_1^2=0$, but it is still instructive, since the interpretations of the curves $G$, $A_i^{(j)}$ in this situation is well known: $D_1$ is a fiber of type $IV^*$ of an elliptic fibration and the curves $G$ and $A_i^{(j)}$ are its components.

	Similarly there is a member of $|2D_1|$ which splits in a the union of a curve $F$ and the union of the curves $A_i^{j}$, where $F\simeq \frac{2D_1-\sum_{j=1}^{n-1}\left(2A_1^{j}+A_2^{j}\right)}{3}$ and hence $2D_1\simeq 3F+\sum_{j=1}^{n-1}\left(2A_1^{j}+A_2^{j}\right)$. Since $F^2=2n+8k$, if $F$ is irreducible and smooth curve, then $g(F)=n+4k+1$. In particular $F$ is rational if $n+4k=-1$ and elliptic if $n=-4k$.
\end{say}

\begin{corollary}\label{cor: pg=2 gen.type}
	There exists a smooth Galois triple cover $X$  of a K3 surface $S$ whose branch locus consists of a smooth curve of genus 4 and 7 $A_2$-configurations of rational curves such that, denoted by $X^o$ the minimal model of $X$, it holds  $$\chi(X^{\circ})=3,\ \ q(X^{\circ})=0,\ \ p_g(X^{\circ})=2,\ \ K_{X^{\circ}}^2=8$$
	\end{corollary}
\proof	Let $S$ be a K3 surface with an elliptic fibration such that the reducible fibers are $I_2+6I_3$ and the Mordell--Weil group is $\Z/3\Z$. The existence of a K3 surface with such an elliptic fibration is guarantee by \cite[Table 1 case 835]{Sh00}. We denote by $F$ the class of the fiber of this fibration, by $\mathcal{O}$ the class of the zero section, by $A_2^{(1)}$ the class of the irreducible component of the fiber $I_2$ which meet the section $\mathcal{O}$ and by $A_h^{(j)}$, $h=1,2$, $j=2,\ldots, 7$ the classes of the two irreducible components not meeting the zero section of the $j$-th reducible fiber, which is a fiber of type $I_3$. We observe that the class of the 3-torsion section $P$, which generates the Mordell--Weil group can be written in terms of the previous curves as
	$$P=2F+\mathcal{O}-\frac{1}{3}(\sum_{j=2}^7A_1^{(j)}+2A_2^{(j)}).$$
	Moreover, we observe that there are 7 disjoint $A_2$-configurations on this surface, which are given by $\mathcal{O}$, $A_2^{(1)}$, $A_h^{(j)}$, $h=1,2$, $j=2,\ldots,7$. Let us consider the divisor $3F+A_2^{(1)}+2\mathcal{O}=D_1$ and we notice that $D_1^2=6$, $D_1F=2$, $D_1A_2^{(1)}=0$, $D_1\mathcal{O}=0$. One can check that $D_1$ is a big and nef divisor and hence in its linear system there is a smooth irreducible curve of genus 4, still denoted by $D_1$. 
	We claim that there exists a triple cover of $S$ branched over $D_1$ and the 7 $A_2$-configurations $\mathcal{O}$, $A_2^{(1)}$, $A_h^{(j)}$, $h=1,2$, $j=2,\ldots,7$. Indeed, the divisor $$L:=\left(D_1+\mathcal{O}+\sum_{i=2}^7 A_1^{(i)}+2\sum_{j=1}^7 A_2^{(j)}\right)/3=3F+A_2^{(1)}+2\mathcal{O}-P$$ is contained in $NS(S)$ and so $$B:=D_1+\mathcal{O}+\sum_{i=2}^7 A_1^{(i)},\ \ C:=\sum_{j=1}^7 A_2^{(j)}, L=(B+2C)/3,\ \ M=(2B+C)/3$$
    form a triple cover data on $S$. So there exists a triple cover $X\ra S$ which satisfies the condition of Proposition \ref{cor: example general type, irreducible D1} with $k=-2$, $n=8$, $d=1$ and then we deduced the properties of $X^o$.\endproof
    
   \begin{rem}{\rm If $S$ is a K3 surface and $p_g(X)=2$, the cover $f:X\ra S$ induces a splitting of the Hodge structure on $T_X$ in a direct sum of two Hodge structures of K3-type (i.e. Hodge structure of weight two of type $(1,\star,1)$), indeed $$T_X=f^*(T_S)\oplus \left(f^*(T_S)\right)^{\perp}.$$
    The Hodge structure of $T_X$ is of type $(2,\star,2)$ and the one of $f^*(T_S)$ is of type $(1,\star',1)$ since it is induced by the Hodge structure of the K3 surface $S$. Hence both $f^*(T_S)$ and $\left(f^*(T_S)\right)^{\perp}$ carry a K3-type Hodge structure.
    
    The surfaces with $p_g=2$ such that the transcendental Hodge structure splits in the direct sum of two K3-type Hodge structure are studied in several context (see e.g. \cite{Mo87}, \cite{Gdouble}, \cite{L19}, \cite{L20}, \cite{PZ19}) and in general it is interesting to look for K3 surfaces geometrically associated to the K3-type Hodge structure. In the case of covers of K3 surfaces $S$ (and in particular in the setting of Corollary \ref{cor: pg=2 gen.type}), at least one of the two K3-type Hodge structures is of course geometrically associated to the K3 surface $S$ indeed it is the pull back of the Hodge structure of $S$.}
\end{rem}

	We give examples of the case $(2)$ of the Theorem \ref{theorem: three cases giusto} such that $D_1$ is reduced and reducible. In particular we consider the case $D_1^2=0$, but $\Lambda_{D_1}$ contains a class with a positive self intersection. In this case the support of $D_1$ consists of a certain number of fibers $F_i$ and some rational horizontal curves $P_j$, such that $$(D_1)^2=(r\sum_iF_i+\sum_jP_j)^2=0.$$
	We denote by $F$ the class of the fiber, therefore $F_iP_j=FP_j$ for all $i$.
	Since $(D_1)^2=\sum_{k=1}^{s}D_1P_k+rFP_k$, $D_1^2=0$ and $FP_k>0$ it follows that 
	$$\exists j\mbox{ such that }D_1P_j<0.$$
	In particular $D_1$ is not nef.
	
	The following lemma shows that each curve $P_j$ such that $D_1P_j<0$ is a section orthogonal to all the other horizontal curves in $D_1$.
	\begin{lemma}\label{lemma: case (2 theorem)} Let $F$ be the class of the fiber of an elliptic fibration and $P_j$'s irreducible horizontal curves. Let $D_1=rF+\sum_{j=1}^kP_j$ be such that $D_1^2=0$ and $D_1$ is reduced.
		
		There exists $j$ such that $D_1P_j<0$ if and only if $FP_j=1$, $P_jP_i=0$ for every $j\neq i$ and $r=1$ \end{lemma}
	\proof We already observed that $D_1$ intersect negatively at least one of its components, say $P_j$. Hence $P_j$ is a fixed component of a non nef divisor and $P_j^2=-2$. So $$D_1P_j=rFP_j+\sum_{i\neq j}P_iP_j-2<0.$$
	So $rFP_j+\sum_{i\neq j}P_iP_j\leq 1$. We observe that $P_j$ is horizontal so $FP_j>0$, and $P_jP_i\geq 0$ because $P_i$ are irreducible effective curves. Hence the only possibility is $r=1$, $FP_j=1$ and $P_iP_j=0$.\endproof

	We recall that by (2) of Corollary \ref{cor: giusto}, if $D_1$ is as above, in the branch locus there are $n-1$ components $D_h$, $n\geq 1$ which are $A_2$-configurations of curves. We denote by $A_1^{(h)}$, $A_2^{(h)}$ the components of such configurations. 
	
	\begin{proposition}\label{prop: case (2) theorem- possibilities}
		Let $S$ be a K3 surface admitting an elliptic fibration with $k$ disjoint sections $P_j$ and whose class of the fiber is $F$ and let $D_1=F+\sum_j{P}_j$. There exists a triple cover $X\ra S$  branched on $D_1$ and other $n-1$ irreducible components $D_h$ if and only if \begin{equation}\label{eq: conditions gen type fibration} D_1=F+\sum_{j=1}^kP_j,\ k\equiv_3 0,\ \ n\equiv_3 1+\frac{k}{3} \end{equation}
		and the data of the triple cover are determined by
		$$B=F+\sum_{h=1}^{n-1}A_1^{(h)},\ \ C=\sum_{i=1}^kP_i+\sum_{h=1}^{n-1}A_2^{(h)}.$$ 
		The surface $X^{\circ}$ is of general type and its numerical invariants are $$\chi(\mathcal{O}_{X^{\circ}})=\frac{60-6n-k}{9}
		,\ \ K_{X^{\circ}}^2=
		\frac{2k}{3}.$$ 
		
	\end{proposition}
	\proof Since $L=(B+2C)/3$ and $LP_i\in\Z$, it follows that if $F$ is a component of $B$, the $P_i$'s must be components of $C$. The divisor $$L=\frac{F+2\sum_{j=1}^kP_j+\sum_{h=1}^{n-1}\left(A_1^{(h)}+2A_2^{(h)}\right)}{3}$$ has to be contained  in $NS(S)$, $LF\in \Z$, $LP_i\in \Z$ for every $i=1,\ldots, h$ and $L^2\in 2\Z$. These conditions implies
	$$\frac{2k}{3}\in\Z,\ \ -1\in\Z\mbox{ and }\frac{-4k-6(n-1)}{9}=2\frac{-2k-3n+3}{9}\in 2\Z.$$
	Recall that $$M=\frac{2F+\sum_{j=1}^kP_j+\sum_{h=1}^{n-1}\left(2A_1^{(h)}+A_2^{(h)}\right)}{3},$$
	so $$L^2=\frac{6-4k-6n}{9},\ M^2=\frac{2k+6-6n}{9}\mbox{ and }LM=\frac{k+3-3n}{9},$$ which gives $\chi(X)$ and $K_X^2$. Since the singularities of $X$
	are negligible, these are the invariants of the minimal resolution of the cover. To obtain a minimal model one has to contract the $(-1)$-curves. Each $A_2$-configuration produces three $(-1)$-curves in the minimal resolution of triple cover and each curve $P_j$ produces one $(-1)$-curve, by Remark \ref{rem: contraction of curves singularities of type 2}. So one contracts at least $3(n-1)+k$ curves to obtain the minimal model from the minimal resolution and so $K_{X^{\circ}}\geq K_X^2+3(n-1)+k=2k/3$.
	
	As in proof of Corollary \ref{cor: giusto irreducible}, if there were other $(-1)$-curves they have to intersect the ramification locus and we already excluded the cases coming from rational curves intersecting the components $D_j$ $j\geq 2$. 
	We consider the canonical resolution of the triple cover  $\widetilde{f}:\tilde{X}\ra \tilde{S}$ as in the diagram \eqref{dia.can}. Thanks to the particular configuration of curves in $D_1$, one checks that the only $(-1)$-curves on $\tilde{X}$ mapped to $\sigma^{-1}(D_1)$ are the strict transforms of the triple cover of the curves $P_i$. After their contraction, there are no other $(-1)$-curves in the strict transform of $\sigma^{-1}(D_1)$. 
	
	Consider a rational curve $C\subset S$ with $CD_1>0$ and observe that $C$ is not mapped to $\sigma^{-1}(D_1)$.  We denote by $\tilde{C}$ the strict transform of $C$. 	
If $\tilde{f}^{-1}(\tilde{C})$ is an irreducible rational curve it intersects the ramification locus in at most 2 points. Moreover $(\tilde{f}^{-1}(\tilde{C}))^2=3\tilde{C}^2\leq -6$. Since we can contract at most two curves meeting $\tilde{f}^{-1}(\tilde{C})$, we can not obtain a $(-1)$-curve. 	If $\tilde{f}^{-1}(\tilde{C})$ splits in three curves and they meet, then they can not be $(-1)$-curves (because $X$ is of general type). If $\tilde{f}^{-1}(\tilde{C})$ splits in three curves these can not meet, $\tilde{f}^{-1}(\tilde{C})$ does not meet the ramification locus and $(\tilde{f}^{-1}(\tilde{C}))^2=(\tilde{C}^2)$. In particular they are not $(-1)$-curves. The $(-1)$-curves on $\tilde{X}$ are contained in the ramification locus and thus $\tilde{f}^{-1}(\tilde{C})$ does not meet these curves. 
	\endproof
	In the previous proposition we do not discuss the existence of the K3 surfaces considered, so a priori it is possible that the hypothesis of the proposition are empty. This is not the case, by the following corollary.
	\begin{corollary}\label{corollary: case (2) theorem: existence} Let $k$ and $n$ be two positive integers such that  $k+2n-1\leq 11$. Then there exists a K3 surface $S$ with an elliptic fibration with $k$ disjoint sections $P_j$ and $n-1$ fibers of type $I_3$ such that the sections $P_j$ all meet the same component of each $I_3$-fiber.
		
		In particular there exists a the triple cover $X\ra S$ as in Proposition \ref{prop: case (2) theorem- possibilities} if $(k,n)=(3,2),(6,3), (9,1)$ and in these cases $\chi(X)=5,4,5$ and $K_{X^{\circ}}^2=2,4,6$ respectively. \end{corollary}
	\proof Let $\Lambda$ be a lattice generated by the following classes $$F,\ P_j,\ (j=1,\ldots, k),\ A_1^{(h)},\ A_2^{(h)},\ \frac{F+2\sum_{i=1}^jP_j+\sum_{h=1}^{n-1}\left(A_1^{(h)}+2A_2^{(h)}\right)}{3}=L,$$ where the intersections which are not zero are $$FP_j=A_1^{(h)}A_2^{(h)}=1,\ P_j^2=\left(A_i^{(h)}\right)^2=-2.$$ The lattice $\Lambda$ is even and hyperbolic. If the rank of the lattice $\Lambda$ is less than 12, then it admits a primitive embedding $\Lambda_{K3}$ and there exists a projective K3 surface $S$ such that $NS(S)=\Lambda$ (by the surjectivity of the period map). If $(k,n)$ is such that $\rk(\Lambda)\leq 12$ and satisfies the condition \eqref{eq: conditions gen type fibration}, then $(k,n)=(3,2),(6,3), (9,1)$. The classes $F$ and $F+P_j$ span a copy of $U$ inside $\Lambda$, then there exists a negative definite lattice $K$ such that $U\oplus K\simeq \Lambda$. In \cite[Proof of Lemma 2.1]{Ko}  it is proved that if $NS(S)\simeq\Lambda$ is as described, there exists an elliptic fibration. The $2(n-1)$ $(-2)$-curves $A_i^{(h)}$, which are roots in $K$, are contained in singular fibers as in the statement. By Shioda--Tate formula, the rank of the Mordell--Weil group of the elliptic fibration is the Picard number of the surface minus the rank of the trivial lattice. The latter is spanned by $U$ and the roots contained in $K$. Hence the Mordell--Weil group has rank $k-1$ and hence there are $k$ independent sections, which corresponds to the classes $P_j$ (see \cite[Corollary 6.13]{SchS}).
	\endproof
\begin{say}	We observe that the K3 surface associated to the values $(9,1)$ corresponds to a K3 surface with an elliptic fibration with nine disjoint sections and, generically, without reducible fibers. This K3 surface is obtained as base change of order 2 on a generic rational elliptic surface $R$. One can directly check this fact by considering the N\'eron--Severi group of $S$, which is generated by $L$ and by the classes $P_1,\ldots, P_9$ and it is isometric to $U\oplus E_8(-2)$ which is the N\'eron--Severi group of a K3 surface corresponding to the double cover of a generic rational elliptic surfaces \cite{GSal}. In particular, the rational elliptic surface is a blow up of $\mathbb{P}^2$ in the base locus of a generic pencil of cubics and hence the K3 surface $S$ can be described as double cover of $\mathbb{P}^2$ branched on two specific cubics of the pencil see e.g. \cite{GSal}. The triple cover $X\ra S$ defines a $6:1$ Galois cover $X\ra\mathbb{P}^2$ whose Galois group is $\mathfrak{S}_3$.
	In particular, the rational surface $R$ admits a non Galois triple cover $W$ and by construction $X$ is a double cover of $W$.\\
	\end{say}

	In Lemma \ref{lemma: case (2 theorem)} and hence in Proposition \ref{prop: case (2) theorem- possibilities} we assume that the fibers appearing in the component $D_1$ of the branch locus are smooth. But of course this is not the only possibility. Indeed one can also consider reducible fibers. This gives many ways to construct the data of the triple cover. For example let $F$ be a reducible fiber with two components $G_0$ and $G_1$, then either both $G_0$ and $G_1$ are contained in $B$ or one is contained in $B$ and the other in $C$. These choices produce different covers and the situation can be easily generalized with fibers with many components.
	
	We now describe one case where the fiber is of type $I_2$ (and then it has two components), one component is contained in $B$ and the other in $C$. Moreover in the branch locus there are 4 sections and other $2$ $A_2$-configurations of rational curves. This means that $n=3$ with the notation of Theorem \ref{theorem: three cases giusto}, case $(2)$.
	
	\begin{example} {\rm 
			Le us consider a K3 surface $S$ and the configuration of $(-2)$-curves as in Figure 2. The existence of a K3 surface with the required fibration is guaranteed by the surjectivity of the period map as in Corollary \ref{corollary: case (2) theorem: existence}.

\begin{center}
\definecolor{qqqqff}{rgb}{0,0,1}
\begin{tikzpicture}[line cap=round,line join=round,>=triangle 45,x=0.7cm,y=0.7cm]
\clip(-5.00,-3.5) rectangle (5.00,7);
\draw [line width=2pt] (-1,-3)-- (-3,-1);
\draw [line width=2pt,color=qqqqff] (-3,-2)-- (-1,0);
\draw [line width=2pt,color=qqqqff] (1,-3)-- (3,-1);
\draw [line width=2pt] (3,-2)-- (1,0);
\draw [samples=50,rotate around={-90:(-0.5,3)},xshift=-0.5cm,yshift=3cm,line width=2pt,domain=-8:8)] plot (\x,{(\x)^2-3/2/1});
\draw [samples=50,rotate around={-270:(0.5,3)},xshift=0.5cm,yshift=3cm,line width=2pt,color=qqqqff,domain=-8:8)] plot (\x,{(\x)^2-3/2/1});
\draw [line width=2pt,color=black] (-3,5)-- (-1,6);
\draw [line width=2pt,color=qqqqff] (3,5)-- (1,6);
\draw [line width=2pt,color=qqqqff] (1,1)-- (3,2);
\draw [line width=2pt,color=black] (-1,1)-- (-3,2);
\begin{scriptsize}
\draw[color=black] (-1.64,-1.57) node {$A^{(1)}_2$};
\draw[color=qqqqff] (-1.64,-0.23) node {$A^{(1)}_1$};
\draw[color=qqqqff] (2.3,-2.29) node {$A^{(2)}_1$};
\draw[color=black] (2.3,-0.53) node {$A^{(2)}_2$};
\draw[color=black] (1.62,4.99) node {$G_1$};
\draw[color=qqqqff] (-1.38,4.99) node {$G_0$};
\draw[color=black] (-1.4, 6.12) node {$P_0$};
\draw[color=qqqqff] (2.28,6.09) node {$P_1$};
\draw[color=qqqqff] (2.28,1.10) node {$P_3$};
\draw[color=black] (-1.72,1.1) node {$P_2$};
\end{scriptsize}
\end{tikzpicture}

		\verb|Figure 3|
		\end{center}
			
			We then consider the triple cover such that 
			$$D_1=G_0+G_1+P_0+P_1+P_2+P_3,\ D_2=A_1^{(1)}+A_2^{(1)},\ \ D_3=A_1^{(2)}+A_2^{(2)}.$$ The triple cover data are
			$$B=G_0+P_1+P_3+A_1^{(1)}+A_1^{(2)},\ C=G_1+P_0+P_2+A_2^{(1)}+A_2^{(2)},\ L=\frac{B+2C}{3},\ M=\frac{2B+C}{3}.$$
			We observe that $\Lambda_{D_1}$ is indefinite (for example $2(G_0+G_1)+P_0$ has a positive self intersection) and $X$ is of general type by Theorem \ref{theorem: three cases giusto}.

			A straight forward calculation shows that
			\( B^2=-10, \quad C^2=-10, \mbox{ and } BC=8.
			\)\\
			Which yields
			\( L^2=-2, \quad M^2=-2, \mbox{ and } ML=0.
			\)
			Since all the singularities in the branch locus are of type 2, and in particular negligible, by Proposition \ref{prop: rito numbers} we obtain  
			\[ \chi(X')=4, \quad  K^2_{X'}=-8.
			\]
			Moreover, since $L^2=M^2=-2$, it follows that $h^1(S,L)=h^1(S,M)=0$, and then  $q(X')=0$. Therefore $p_g(X')=3$.
			The surface $X'$ is smooth but not minimal. Indeed each configuration of type $A_2$ in the branch locus corresponds to three $(-1)$-curves on the minimal resolution $X'$ of the cover and each curve $P_j$ corresponds to a $(-1)$-curve on $X'$, see Remark \ref{rem: contraction of curves singularities of type 2}. So we have to contract at least $3\cdot 2+4=10$ curves on $X'$ and we denote by $X_{m}$ the surfaces obtained contracting these 10 curves on $X'$. Then $K_{X_{m}}^2=-8+10=2$. The surface $X_{m}$ is minimal, to prove that directly it is not straightforward. Nevertheless, it is possible to see it using a different construction of $X_{m}$. In \cite{BP17}, the second author with Bini introduce a Calabi-Yau threefold $Y_6$ with Hodge numbers (10, 10). To some extent, this is special. In fact, its group of automorphisms contains a subgroup $G$ isomorphic to $\ZZ/6$. Moreover, this Calabi-Yau threefold can be described as a small resolution of a (3, 3) complete intersection $Y$ in $\PP^5$ with $72$ ordinary double points.  Furthermore, the group $G$ extends to a group of automorphisms of $\PP^5$. Thus, there are six invariant hyperplane sections corresponding to irreducible characters of the abelian group $\ZZ/6$. Since the intersections of $Y$ with these sections are invariant with respect to this group, it was natural to investigate them and their quotients. 
Out of the six invariant sections mentioned before, three of them are irreducible. These are singular surfaces of general type; on the minimal model $\Sigma$ of one of them let act $\ZZ/2 \leq G$, then it is not hard to prove that the minimal resolution of $\Sigma/(\ZZ/2)$ is a minimal surface and is isomorphic to $X_{m}$. Hence $X_{m}$ is the minimal model of $X$, i.e. it is $X^{\circ}$. 	
}\end{example}
	
	\subsection{Examples of case $(3)$ of Theorem \ref{theorem: three cases giusto}}\label{subsec: base change of elliptic fibration}

	The case $(3)$ of Theorem \ref{theorem: three cases giusto} implies that $D_1$ is a fiber of an elliptic fibration. So $D_j$, $j>1$, is necessarily contained in a fiber and this naturally gives two cases: either for all $j=1,\ldots, n$ $D_j$ is a full fiber or at least one of the $D_j$ is supported on rational components of a fiber but it does not coincide with the full fiber. Both cases are possible and we now discuss them. 
	
	\begin{proposition}
		Let $X\ra S$ be a triple cover as in case $(3)$ of Theorem \ref{theorem: three cases giusto}. This implies that  $\varphi_{|D_1|}:S\ra \mathbb{P}^1$ is an elliptic fibration. Suppose that all the $D_j$ are linearly equivalent to $D_1$, then $X$ is obtained by a base change of order 3.
	\end{proposition}
	\proof By proof of Corollary \ref{cor: giusto}, $D_1$ is the fiber of an elliptic fibration $\varphi_{|D_1|}:S\ra\mathbb{P}^1$. By hypothesis the $D_j$'s are fibers of the fibration $\varphi_{|D_1|}$ too. The $3:1$ cover $X\ra S$ induces a $3:1$ cover $C\ra\mathbb{P}^1$ where $C$ is a smooth curve. The branch locus of $C	\ra \mathbb{P}^1$ is the image of the branch locus of $X\ra S$ and so it is $\varphi_{|D_1|}(\bigcup_j D_j)$.\endproof

	\begin{proposition}\label{prop: examples case 3 of proposition} Let $S$ be a K3 surface endowed with an elliptic fibration $\varphi_{|F|}:S\ra\mathbb{P}^1$. Let us consider $p_1,\ldots, p_n$ points in $\mathbb{P}^1$ and a triple Galois cover $g:W\ra \mathbb{P}^1$ totally branched on $p_1,\ldots, p_n$. The fiber product $X:=S\times_g\mathbb{P}^1$ is a triple cover of $S$ branched on $n$ fibers. If the fibers over the points $p_i$ are reduced, $X$ is also normal. Denoted by $X^\circ$ the minimal model of $X$, if $X^{\circ}$ is not a product, then  $$h^{1,0}(X^{\circ})=g(W)=n-2\mbox{ and } \ K_{X^{\circ}}^2=0.$$ 
If $X^{\circ}$ is a product, then $h^{1,0}(X^{\circ})=g(W)+1=n-1$.
 
		If all the branch fibers are reduced and not of type $IV$, then $$e(X^\circ)=72,\ \ \chi(X^\circ)=6 \mbox{ and }p_g(X^{\circ})=3+n.$$ 
	\end{proposition}
	
	\proof 
	The cover automorphism of $g:W\ra\mathbb{P}^1$ acts as $\zeta_3$ locally near the first $b_1$ ramification points and as $\zeta_3^2$ locally near the other $b_2=n-b_1$ ramification points. Notice that $b_1+2b_2\equiv_3 0$. Let us consider the fiber product $S\times_g\mathbb{P}^1$. It is the triple cover $X$ of $S$ branched over the fibers $F_{p_i}=\varphi_{|F|}^{-1}(p_i)$. So we have a Galois triple cover $X\ra S$. If all the fibers $F_{p_i}$ are reduced, the cover $X$ is normal and we can apply the previous theory. The branching data are $B=\sum_{i=1}^{b_1}F_{p_i}$, $C=\sum_{j=b_1+1}^{n}F_{p_j}$, $L\simeq (b_1+2b_2)F/3$ and $M\simeq (2b_1+b_2)F/3$.
	
	Now let us compute the triple cover invariants. The surface $X$ is endowed with the elliptic fibration. Let us denote by $X^{\circ}$ the minimal model of $X$, which is a smooth surface, admitting a relatively minimal elliptic fibration $\mathcal{E}:X^\circ\ra W$. Assume that it admits at least one singular fiber (which is surely true if there exists a singular fibers of $S$ which is not in the branch locus). Then $X^{\circ}$ is not a product and $h^{1,0}(X^{\circ})=g(W)=n-2$. Moreover, since $X^{\circ}$ admits an elliptic fibration $K_{X^{\circ}}^2=0$. We recall that if a branch fiber are of type $IV$ the corresponding ramification fiber on $X^{\circ}$ is smooth, hence it is possible that after the base change the fibration $X^{\circ}\ra W$ has no singular fibers and it maybe a product.
	
	If the branch fibers are reduced and different from $IV$, the Euler characteristic of their strict transform is 3 times their Euler characteristic (see \cite[VI.4.1]{M85}), then $e(X^{\circ})=3e(S)=72$ so $\chi(X^{\circ})=e(X^{\circ})/12=3\chi(S)=6.$
	
		\endproof
		
We observe that even if there are non reduced fibers $F_{p_i}$ in the branch locus, the invariants of the minimal model of a normalization can be computed by the theory of the base change of elliptic fibrations, see \cite[VI.4.1]{M85}.

	\begin{theorem}\label{theo: list negligible}\label{say: other singularities in the branch locus} The  singularities in the branch locus of a Galois triple cover are negligible (see Definition \ref{def.neg.sing}) if the local equation of the branch locus is in the following list: \begin{itemize}\item $xy$ (type 1)\item $xy^2$ (type 2);
			\item $xy(x+y)$ (simple triple point); \item $x^2-y^3$ (cusp); \item $x(y-x^2)$\end{itemize}\end{theorem}	
	\proof The singularities of type 1 and 2 were considered in \cite[Example 1.6 and 1.8]{PP13} (see also Proposition \ref{prop: negligible}). For the other cases we use the results of Proposition \label{prop: examples case 3 of proposition}, where we computed the invariants of $X^{\circ}$ directly by considering the geometry of the elliptic fibration.  We compare them with the ones obtained by applying the point $(3)$ of Proposition \ref{prop: rito numbers}. If they coincide, this means that all the singularities which can appear in the branch locus are negligible.
	Since $L^2=M^2=LM=0$, one obtains $\chi(X')=6$, $e(X')=72$ and $K_{X'}^2=0$. Moreover, $h^1(L)=h^1(\frac{b_1+2b_2}{3}F)=\frac{b_1+2b_2}{3}-1=\frac{b_1+2b_2-3}{3}$ and $h^1(M)=h^1(\frac{2b_1+b_2}{3}F)=\frac{2b_1+b_2}{3}-1=\frac{2b_1+b_2-3}{3}$ so that $h^1(L)+h^1(M)=b_1+b_2-2=n-2$. So all the singularities appearing in the singular reduced fibers are negligible.
	
	For example, a branch fiber of type $III$ in $S$ consists of two tangent rational curves. We deduce that the singularities in the branch locus obtained by tangency of two components are negligible singularities. More precisely if there is a branch fiber of type  $III$ on $S$, it induces a fiber of type $IV^*$ (whose dual graph of curves is $\widetilde{E}_6$) of $\mathcal{E}:X'\ra W$. We deduce that if the branch locus of a totally ramified triple cover contains two tangent curves, the triple cover has a singularities of type $E_6$ (see \verb|Figure 4|).

	\begin{center}
		\includegraphics[width=8cm]{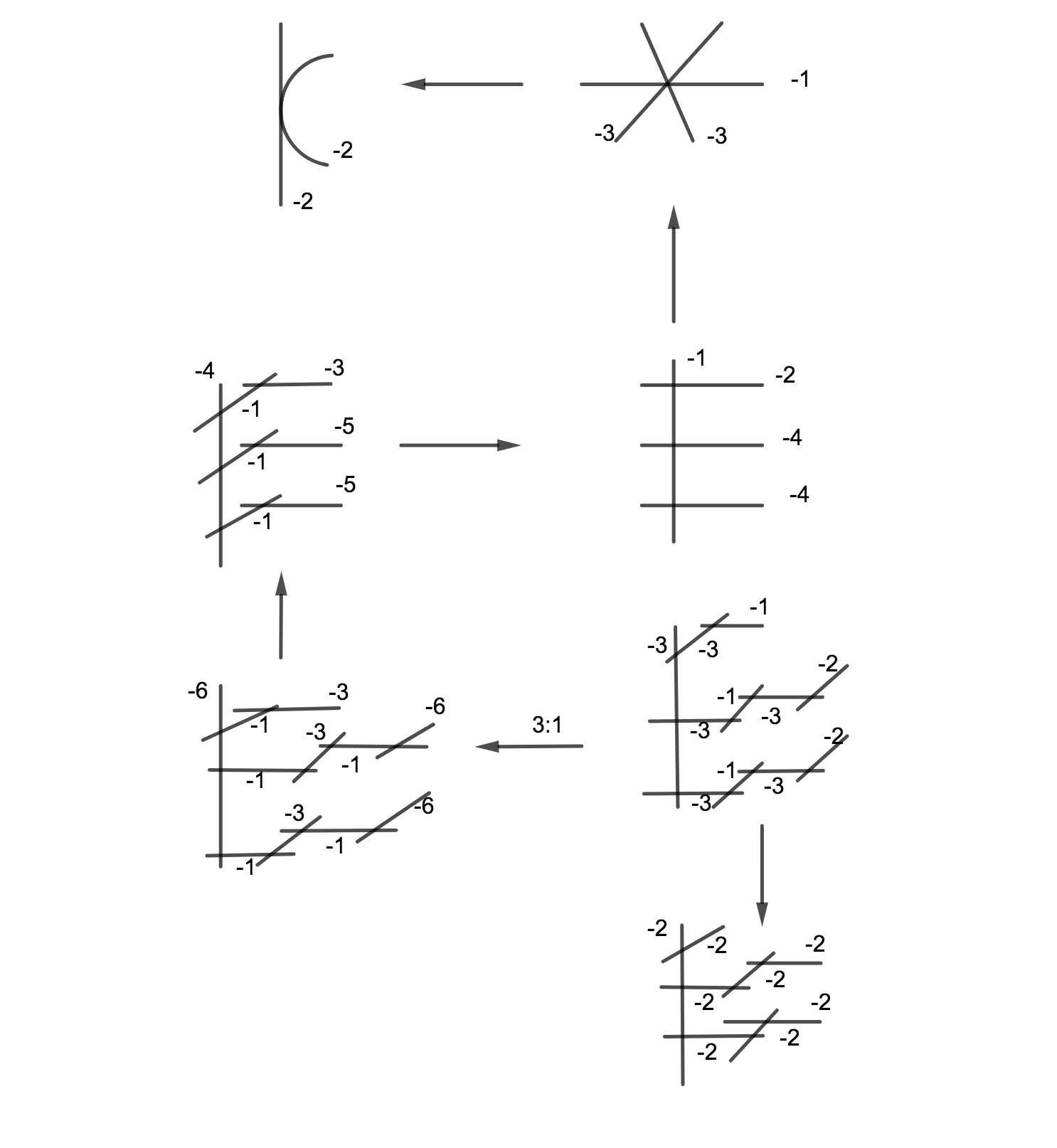}
		
		\verb|Figure 4|
	\end{center}
	
	If a branch fiber on $S$ is singular, then the corresponding fiber on $X^{\circ}$ is obtained by a base change of order 3;  \cite[Table VI.4.1]{M89} shows that effect of a base change on singular fibers of an elliptic fibration.

	By \cite[Table VI.4.1]{M89} one immediately obtains that, if the branch contains a cusp, the cover has a singularity of type $D_4$, if the branch has a simple triple point, the cover has an elliptic singularity.
	Notice that one recovers the singularities of type 1 by considering fibers of type $I_m$ in the branch locus. \endproof
	
	Note that the list of Theorem \ref{theo: list negligible} is not necessarily complete.\\
	
\medskip

	We now consider the other possibility of $(3)$ of the Theorem \ref{theorem: three cases giusto}, hence we assume that $D_1$ is a fiber but at least one of the other components of the branch locus is strictly contained in a fiber.
	
	In this case, the effect of the triple cover on the fibers strictly containing $D_j$ is not the one of a base change of order 3. In  the following proposition we describe how the fibers changes under a triple cover of this type.
	\begin{corollary}
		Let $D_1$ be a connected reducible (possibly non reduced) component of the branch locus of a Galois triple cover $f:X\ra S$. Let  $\varphi_{|D_1|}:S\ra \mathbb{P}^1$ be the induced elliptic fibration (see proof of Theorem \ref{theorem: three cases giusto}). Let $D_j\subset (F_S)_j$ and we assume that $D_j=(F_S)_j$ iff $j\leq k$. Given a fiber $F_S$, we denote by $F_{X'}$ the correspondent fiber in the minimal model of the normalization of $X$.
		Then we have the following tables of types of singular fibers:
		$$\begin{array}{cc}
			\begin{array}{|c|c|}
				j \leq k&\\
				\hline
				\mbox{type }F_S&\mbox{type }F_{X'}\\
				\hline 
				I_m&I_{3m}\\ 
				\hline 
				I_m^*&I_{3m}^*\\ 
				\hline 
				II^*&I_0^*\\ 
				\hline 
				III^*&III\\ 
				\hline 
			
				II&I_0^*\\ 
				\hline 
				III&III^*\\ 
				\hline 
				IV&I_0\\ 
				
			\end{array}&\begin{array}{|c|c|}j>k&\\
				\hline
				\mbox{type }F_S&\mbox{type }F_{X'}\\
				\hline 
				I_{3m}&I_{m}\\
				\hline
				IV&IV\\
				\hline
				IV^*&I_0\end{array}\end{array}$$
		
	\end{corollary}
	\proof
	If $j\leq k$, then $D_j=(F_S)_j$, the fibers $(F_S)_j$ are branch fibers and the type of $(F_X)_j$ is given in \cite{M85}. We already observed that if $D_i$ is properly contained in a fiber $F_S$, $i\geq k$ and $D_i$ is supported on an $A_2$-configuration.
	
	Let us suppose that $F_S$ properly contains $r$ connected components $D_i$, $i\geq k$. These are $A_2$-configuration, whose components are denoted by $A_1^{(1)}$, $A_2^{(1)}$, $A_1^{(2)}$, $A_2^{(2)}$, $\ldots$, $A_1^{(r)}$, $A_2^{(r)}$.

	Moreover, $\sum_{h=1}^r(A_1^{(h)}+2A_2^{(h)})/3$ necessarily has an integer intersection with all the components of $F_S$, hence it is contained in the discriminant group of lattice associated to $F_S$. We conclude that $F_S$ is necessarily one of the following: $I_{3m}$, with $r=m$, $IV$ with $r=1$ and $IV^*$ with $r=2$. 
	
	There are $m$ $A_2$-configurations contained in $I_{3m}$ and the birational inverse image of each of them in the minimal model $X^{\circ}$ is a point. The remaining curves form an $I_m$ fiber.
	
	To obtain the minimal model in case $IV^*$, one first contracts the inverse image of the curves in the $A_2$-configurations each to a point. These three points lie on a $(-1)$-curve (which is the inverse image of the unique remaining curve in $IV^*$). Contracting this curve we obtain an $I_0$ fiber.
	
	The case $IV$ is as the {\bf{\verb|Figure 5|}}.
	
	\begin{center}
		\includegraphics[width=8cm]{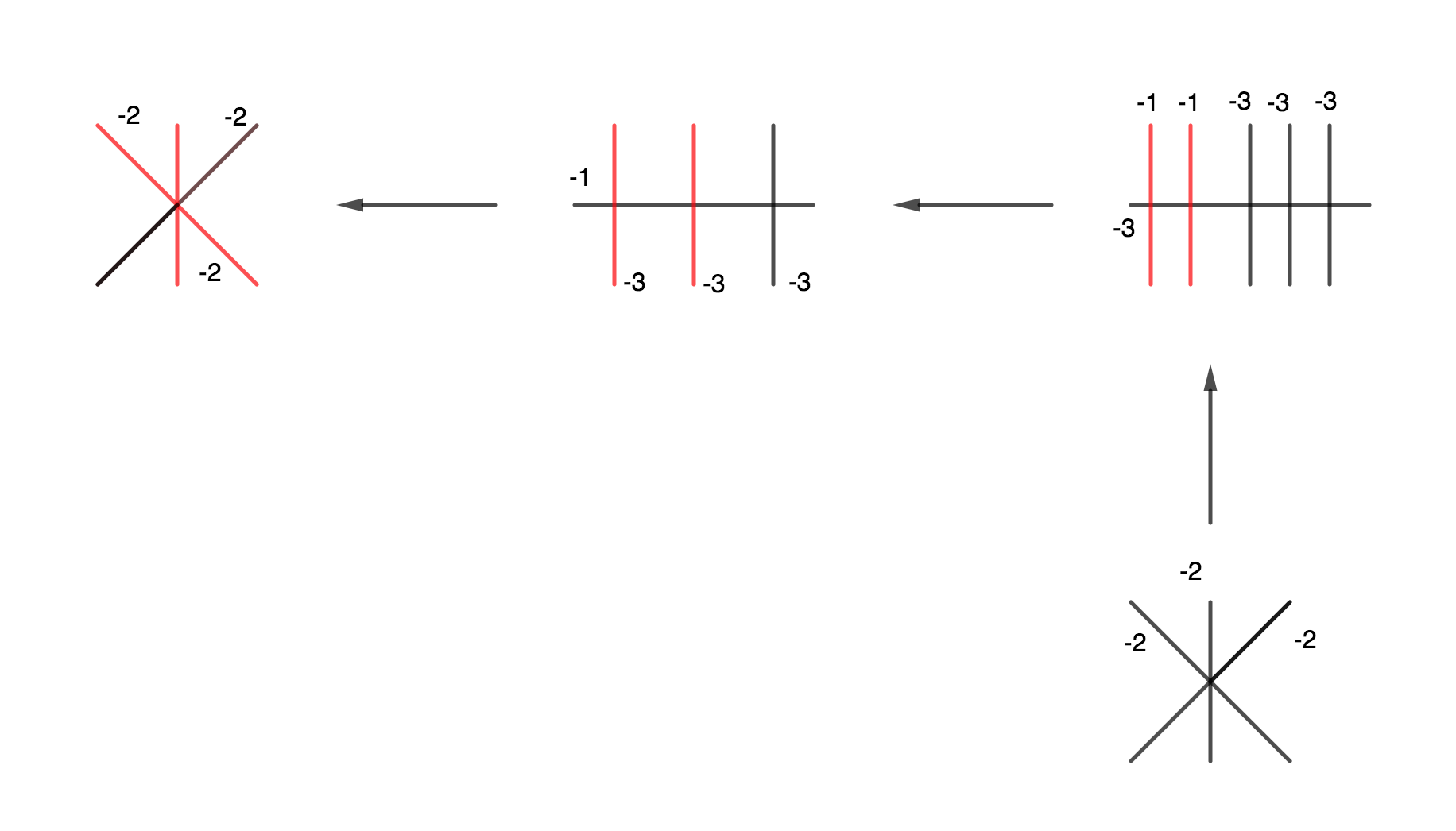}
		
		\verb|Figure 5|
	\end{center}
	\endproof

	\subsection{Irregular covers of K3 surfaces}\label{subsec: irregular}

	Even if a K3 surface is a regular surface, it is of course possible to construct triple covers of K3 surfaces which are irregular surfaces. The easiest examples are provided by cases $(1)$ and $(3)$ of Theorem \ref{theorem: three cases giusto}. Indeed, by case (1) of Theorem \ref{theorem: three cases giusto} the Galois triple cover of a K3 surface branched on 9 $A_2$-configuration is an Abelian surface. This case is effective, since it is known that there exist Abelian surfaces admitting a symplectic automorphism of order 3 such that their quotient by this automorphism is birational to a K3 surface, \cite{F}.
	In Proposition \ref{prop: examples case 3 of proposition} the irregularity of a surface $X$ obtained as base change of order 3 on an elliptic fibration on a K3 surface $S$ is computed. One check that if $n>2$, then the surface $X$ is irregular. Once again it is clear that there exists explicit examples of this situation, indeed it suffices to construct a curve $W$ (with the notation of Proposition \ref{prop: examples case 3 of proposition}) which is a Galois triple cover of $\mathbb{P}^1$ branched on more than 2 points.
	
	We observe that the previous constructions produce surfaces $X$ whose Kodaira dimension is either 0 or 1.
	It is more complicated to construct examples of irregular covers of K3 surfaces which are surfaces of general type. This is essentially due to the difficulties in finding branch divisors on a K3 surfaces which are not supported on rational or elliptic curves, but such that the associated triple cover is irregular. Indeed if $X \ra S$ is a triple cover of a K3 surface, than $X$ is irregular if and only if at least one between $L$ and $M$ satisfies $h^1(L)>0$ or $h^1(M)>0$, and there are not so many divisors with this property.

	\begin{lemma}\label{lemma: h1D neq 0}
		Let $D$ be a divisor on a K3 surface $S$ such that $-D$ is not effective. If $h^1(D)\neq 0$ then one of the following holds
		\begin{itemize}
			\item if $D^2<0$ then $D^2\leq -4$, and if $D^2=-4$, then $D$ is effective;
			\item if $D^2= 0$ and $D$ is nef, then $D=nE$ where $E$ is a genus 1 curve, $n>1$ and $h^1(S,D)=n-1$; if $D^2=0$ and $D$ is not nef $D=M+F$ where $M$ is its moving part and $F$ is its fix part, and $$F(F-2D)>0.$$ 
			\item if $D^2>0$, then $D$ is not nef. In this case $D=M+F$ where $M$ is its moving part and $F$ is its fix part, and $$F(F-2D)>0.$$ 
		\end{itemize}
	\end{lemma}
	\proof If $D^2<0$, then $h^0(S,D)\leq 1$. If $D^2=-2$, then $\chi(D)=1$ which would imply $h^1(S,D)=0$.
	
	\medskip
	
	If $D^2>0$ and $D$ is nef, then $D$ is big and nef, and by Kawamata--Viehweg vanishing theorem $h^1(S,D)= 0$. Therefore, if $D^2>0$ and $h^1(D)\neq 0$ then $D$ is not nef. In particular, there is a fixed part $F$ in $|D|$ such that $D=M+F$ with $F^2<0$ and $DF<0$. Moreover $h^0(S,D)=h^0(S,D-F)$. By Riemann--Roch theorem  
	$$\frac{1}{2}(D-F)^2+2=h^0(S,D-F)=h^0(S,D)=\frac{1}{2}D^2+2+h^1(S,D).$$

	\medskip
	
	Finally, if  $D^2=0$ and $D$ is not nef the proof is the same as the case above. If $D$ is nef, then $D=nE$ where $E$ is a genus 1 curve. We recall the well known fact $h^1(S,nE)=n-1$. To recover this, note that $L^2 = 0$ implies $h^0(S, L) \geq 2$, as by Serre duality $h^2(S, L) = h^0(S, L^{-1}) = 0$ for
	the non-trivial nef line bundle $L$ (intersect with an ample curve). It is enough to use Riemann--Roch theorem for a K3 surface to compute $\chi(\oo_S(nE))=2$ and then use inductively the exact sequence in cohomology associated to the fundamental sequence of $E$ tensorized with $\oo_S(nE)$:
	\[
	0 \ra \oo_S\big((n-1)E\big) \ra \oo_S(nE) \ra \oo_C(nE|_E) \ra 0.
	\]
	\endproof

\begin{say}	In view of Lemma \ref{lemma: h1D neq 0}, we consider divisors on a K3 surface with very low self intersection,  In Theorems \ref{theo: cover S16} and \ref{theo: cover S15}  we construct specific irregular surfaces of general type which are covers of K3 surfaces and some notation are needed.
	
	We denote by $M_{(\Z/2\Z)^4}$ a specific overlattice of $\oplus_{1}^{15}A_1$, constructed as follows:
	denoted by $N_i$ the 15 orthogonal classes with self intersection $-2$ which generate $\oplus_{1}^{15}A_1$, we add to the lattice $\langle N_i\rangle$ the vectors
	\begin{eqnarray*}\begin{array}{c}v_1:=\left(\sum_{i=1}^{8}N_i\right)/2, v_2:=\left(\sum_{i=1}^{4}N_i+\sum_{j=9}^{12}N_j\right)/2\\
			v_3:=\left(N_1+N_2+N_5+N_6+N_9+N_{10}+N_{13}+N_{14}\right)/2,\ v_4:=\left(\sum_{i=0}^{7}N_{2i+1}\right)/2.\end{array}\end{eqnarray*}
	The lattice $M_{(\Z/2\Z)^4}=\langle N_1,\ldots ,N_{15},\ v_1,\ldots, v_4\rangle$ is an even negative definite lattice of discriminant $(\Z/2\Z)^7$.

	Similarly, we define $M_{(\Z/2\Z)^3}$ the overlattice of $\oplus A_1^{14}$  obtained as above starting with 14 divisors $N_i$ and adding the classes $v_1$, $v_2$, $v_3$. The lattice $M_{(\Z/2\Z)^3}$ is an even negative definite lattice of discriminant $(\Z/2\Z)^8$.

	Let us consider the rank 16 lattice $R_{16}:=\langle 6\rangle\oplus M_{(\Z/2\Z)^4}$ and let us denote by $H$ the generator of $\langle 6\rangle$. By \cite[Theorem 8.3]{GS}, there exists an even overlattice $R'_{16}$ of index 2 of $R_{16}$ obtained by adding to the lattice $R_{16}$ the class $\left(H+\sum_{i=1}^{15}N_i\right)/2$. The discriminant group of $R'_{16}$ is $\Z/6\Z\times (\Z/2\Z)^5$. By \cite[Theorem 1.14.4 and Remark 1.14.5]{Nik Int} $R'_{16}$ admits a primitive embedding in $\Lambda_{K3}$. So there is a K3 surface (indeed a 4 dimensional family of K3 surfaces) whose N\'eron--Severi group is isometric to $R'_{16}$ lattice. Observe that this K3 surface has a model (given by the map $\varphi_{|H|}$) as the complete intersection of a quadric and a cubic in $\mathbb{P}^4$ with 15 nodes. \end{say}
	\begin{theorem}\label{theo: cover S16}
		Let $S_{16}$ be a K3 surface such that $NS(S_{16})\simeq R'_{16}$. On the surface $S_{16}$ there are a smooth curve $G$ of genus 4 and fifteen disjoint rational curves $N_i$, $i=1,\ldots, 15$  such that $GN_i=1$, $i=1,\ldots, 15$. There exists a Galois triple cover $\pi:X\ra S_{16}$ branched on $G\bigcup\cup_i N_i$. The invariants of minimal model $X^{\circ}$ of $X$ are:
		$$p_g(X^{\circ})=6,\ q(X^{\circ})=1,\ c_1(X^{\circ})^2=18$$\end{theorem}
	\proof By \cite[Proposition 8.5]{GS}, the divisors $H$ and $\left(H-\sum_{i=1}^{15}N_i\right)/2\in NS(S_{16})$ are pseudoample. So $(3H-\sum_{i=1}^{15}N_i)/2\in NS(S_{16})$ is a pseudoample divisor whose self intersection is $6$. Hence there is a smooth curve of genus 4, denoted by $G$, in $\left|\left(3H-\sum_{i=1}^{15}N_i\right)/2\right|$. Moreover, one can assume that the divisors $N_i$ are supported on irreducible rational curves (see \cite[Propositions 2.3 and 5.1]{G}). 
	So, there are a smooth genus 4 curve $G$ and 15 disjoint rational curves $N_i$ on $S$ such that $GN_i=1$ for $i=1,\ldots, 15$. Set:
	$$B:=G,\ \ C:=\sum_{i=1}^{15}N_i,\ \ L:=\frac{2B+C}{3}=H, \ \ M:=\frac{B+2C}{3}=\left(H+\sum_{i=1}^{15}N_i\right)/2.$$ The divisors $B$, $C$, $L$ and $M$ satisfy the conditions in \ref{say_tripleGalois} and thus determine a triple cover of $S_{16}$.
	The branch locus is not smooth and we consider the minimal resolution $X'$ of $X$. The singularities of the branch locus are negligible, since they are transversal intersection of smooth curves, and thus the invariants of $X'$ are obtained applying the formulae given Proposition \ref{prop: rito numbers}.
	Since $L^2=6$, $M^2=-6$ and $LM=3$, one obtains $\chi(\mathcal{O}_{X'})=6$, $K^2_{X'}=3$ and $e(X')=69$.
	The surface $X'$ is non minimal and the inverse images on $X$ of the curves $N_i$ are 15 disjoint exceptional curves $E_i$. We now prove that these are the unique $(-1)$-curves also even after their contraction. Indeed suppose that there is a $(-1)$-curve on a contraction of $X'$ which is mapped to a curve $I$ on $S_{16}$ with $I\neq N_i$. Then $I$ is a rational curve and in the triple cover it has to split and intersect the exceptional curves $E_i$. Otherwise the self intersection of the inverse image of $I$ is lower than $-1$, cfr. proof of Proposition \ref{prop: case (2) theorem- possibilities}. By direct inspection one sees that also this case is not possible, since $I$ has to split in three $(-1)$-curves which meet.
	So the minimal model $X^{\circ}$ is obtained by $X'$ contracting the curves $E_i$.
	So one obtains  $K^2_{X^{\circ}}=18$ and $e(X^{\circ})=54$ and $\chi(X)=\chi(X^{\circ})$.
	Moreover, one has $$h^{1,0}(X^{\circ})=h^{1,0}(X)=0+h^1(S,L)+h^1(S,M).$$

	Since $L=H$ is a pseudoample divisor, $h^1(S,L)=0$. Since $M^2=-6$ and $M$ is not effective, it follows $h^1(S,M)=1$. One concludes that $q(X^{\circ})=q(X)=1$ and $p_g(X^{\circ})=p_g(X)=\chi(X^{\circ})=6$.\endproof
	
	Let us now consider the rank 15 lattice $R_{15}:=\langle 4\rangle\oplus M_{(\Z/2\Z)^3}$ with $H$ the generators of $\langle 4\rangle$. Let $R'_{15}$ be the overlattice of $R_{15}$ constructed by adding the class $v=(H-\sum_{i=1}^{14}N_1)/2$. The lattice $R'_{15}$ admits a primitive embedding in $\Lambda_{K3}$ and thus there exists a K3 surface $S_{15}$ whose N\'eron--Severi group is isometric to $R'_{15}$.
	
	\begin{theorem}\label{theo: cover S15}
		Let $S_{15}$ be a K3 surface such that $NS(S_{15})\simeq R'_{15}$. On the surface $S_{15}$ there are a smooth curve $G$ of genus 2 and fourteen disjoint rational curves $N_i$, $i=1,\ldots, 14$  such that $GN_i=1$, $i=1,\ldots, 14$. There exists a Galois triple cover $\pi:X\ra S_{15}$ branched on $G\bigcup\cup_i N_i$. The invariants of the minimal model $X^{\circ}$ of $X$ are:
		$$p_g(X^{\circ})=4,\ q(X^{\circ})=1,\ c_1(X^{\circ})^2=12.$$\end{theorem}
	\proof The proof is analogous to the one of the previous proposition, but one has to chose $G$ in the linear system $\left|3H-\sum_{i=1}^{14}N_i)/2\right|$. So one finds $$B:=G,\ \ C:=\sum_{i=1}^{14}N_i,\ \ L:=\frac{2B+C}{3}=H, \ \ M:=\frac{B+2C}{3}=\left(H+\sum_{i=1}^{14}N_i\right)/2,$$ thus $$L^2=4,\ \ M^2=-6,\ \ LM=2  $$ and one has to contract 14 curves to obtain $X^{\circ}$ from the minimal resolution $X'$ of $X$. 
	So $K_{X^{\circ}}^2=K_{X'}^2+14=-2+14=12$, $\chi(X^{\circ})=\chi(X')=5$. As in the previous proof one finds $q(X^{\circ})=1$.\endproof

	A different idea for finding irregular triple cover is to exploit the Albanese morphism and the Kummer construction, but the following remark shows that this approach is too naive.
	
	\begin{remark}{\rm Due to the relation of an Abelian surface $A$ and its Kummer surface $Km((A))$ it is possible that a Galois triple cover $Y\ra A$ defines a triple cover $X\ra Km(A)$. In particular this happens if the involution $\iota: a\ra -a$ on $A$ induces an involution on $Y$. In this case one has the following diagram 
			$$\xymatrix{Y\ar[r]^{3:1}\ar@{-->}[d]^{2:1}&A\ar@{-->}[d]^{2:1}\\X\ar[r]^{3:1}&Km(A)}$$
			Since $q(A)=2$ it is easier to find cover $Y$ such that $q(Y)\neq 0$.
			Nevertheless let $Y$ be a surface such that the Albanese variety $Alb(Y)$ coincides with $A$ (so in particular it has dimension 2 and the Albanese map is $3:1$), then it holds: 
			
			If $\iota$ lifts to an involution $\iota_Y$ of $Y$, the quotient $X:=Y/\iota_Y$ is a regular surface.
			
			Indeed the Albanese surface $A$ is defined as $H^0(Y,\Omega_1)^{\vee}/H_1(Y,\Z)$ and $\iota$ acts on the space $H^0(Y,\Omega_1)$ as $-1$. Thus $\iota_Y$ does not preserve the 1-holomorphic form of $Y$.}
		
	\end{remark}

	\section{Triple cover of K3 surfaces: the split not Galois case}\label{sec: triple split non Galois cover}
	
	Now we analyse the non-Galois triple covers  $f\colon X\ra S$  of a K3 surface under the assumption that the Tschirnhausen bundle $\EE$ splits. We will provide an example of a such a triple cover for all possible Kodaira dimension. Let $\EE$ be a direct sum of two line bundles $\mathcal{L}^{-1}$ and $\mathcal{M}^{-1}$ so that $\EE^{\vee}=\mathcal{L}\oplus\mathcal{M}$.
	We have already observed following diagram \eqref{diag split triple cover}
	that the triple cover $f$ is totally branched over $D'$ and simply branched over $D''$. All the other covers in the diagram are Galois covers.\\
	
	Since $\EE^{\vee}=\mathcal{L}\oplus \mathcal{M}$, we have  that 
	$$2L+2M=2D'+D''\mbox{ is the branch locus of the triple cover} f:X\ra S,$$ with $D''\neq 0$ and there exists four effective divisors $B'$, $C'$, $B''$ and $C''$ such that $$L=\frac{2B'+C'}{3}+\frac{B''}{2},\ \ M=\frac{B'+2C'}{3}+\frac{C''}{2}.
	$$
	So $L+M=B'+C'+\left(B''+C''\right)/2$ is the branching data of an $\mathfrak{S}_3$ cover of $S$, i.e. the line bundle $\mathcal{O}_S(L+M)$ returns the geometric line bundle $\mathbb{L}$ of the \cite[Theorem 6.1]{CP}. This $\mathfrak{S}_3$-cover is the Galois closure of the non Galois triple cover $X\ra S$. Notice that in \cite[Theorem 6.1]{CP} one assumes the branch locus to be smooth, but the results extend also to the non-smooth case.
	
	\begin{say}\label{say:non Galois split K3 K3}{\bf A split non Galois triple cover of a K3 surface with a K3 surface.}
		
		We construct a split but not Galois triple cover $f:X\ra S$ such that $S$ is a K3 surface and $X$ is a singular surface whose minimal resolution is still a K3 surface. We refer to the diagram \eqref{diag split triple cover} for the notation.
		
		Let $Z$ be a K3 surface such that $\mathfrak{S}_3\subset \Aut(Z)$ and $\mathfrak{S}_3$ acts symplectically on $Z$. Then the quotient surface $Z/\mathfrak{S}_3$ has 3 singularities of type $A_2$ and 8 singularities of type $A_1$, see \cite[ p.78 case 6]{X}. 
		
		The resolution of $Z/\mathfrak{S}_3$ is a  K3 surface $S$, which admits a Galois $\mathfrak{S}_3$ cover, branched on the Jung-Hirzebruch strings which resolve the singularities. Let us denote by $B_i\cup C_i$, $i=1,2,3$ the $i$-th $A_2$ configuration and by $N_j$, $j=1,\ldots, 8$ the $j$-th $A_1$ configuration.
		Then $L=\left(\sum_{i=1}^32B_i+C_i\right)/3+\sum_{j=1}^4\left(N_j\right)/2$ and $M=\left(\sum_{i=1}^3B_i+2C_i\right)/3+\sum_{j=4}^8\left(N_j\right)/2.$
		The K3 surface $S$ admits a $2:1$ cover branched along $\cup_{j=1}^8 N_j$, which is the non-minimal surface $W$, whose minimal model is a K3 surface $W^{\circ}$. There are 6 $A_2$-configurations on $W^{\circ}$, inverse image of the 3 $A_2$-configurations in $S$. These 6 $A_2$-configurations form a 3-divisible set.
		The Galois triple cover of $W^{\circ}$ branched on these 6 $A_2$-configurations is a non minimal surface, whose minimal model is $Z$. The quotient of $Z$ by an involution in $\mathfrak{S}_3$ is the singular surface $X$, whose minimal resolution is another K3 surface, $X^{\circ}$. The surface $X$ is by construction the non Galois triple cover of $S$ associated to $\mathcal{E}^{\vee}=\mathcal{L}^{-1}\otimes\mathcal{M}^{-1}$.
		The total ramification of $S$ is on $D'=\sum_{i=1}^3(B_i+C_i)$ and the simple one is on $D''=\sum_{i=1}^8 N_i$.

	\end{say}
	
	\begin{say}\label{say:non Galois split K3 Kod=1}{\bf A split non Galois triple cover of a K3 surface with a properly elliptic surface.}
		
		Let us consider a K3 surface $S$ endowed with an elliptic fibration $\mathcal{E}:S\ra \mathbb{P}^1$. Let us consider $g:C\ra\mathbb{P}^1$ a split non Galois triple cover of $\mathbb{P}^1$. This can be constructed by considering $2k$ points $P_{i}$, $i=1,\ldots, k$ and $2h$ points $Q_j$, $j=1,\ldots, 2h$.
		
		Then one uses, as triple cover data, $B'=\sum_{i=1}^kP_i$, $C'=\sum_{i=k+1}^{2k}P_i$, $B''=\sum_{j=1}^{2r}Q_j$, $C''=\sum_{j=2r+1}^{2h}Q_j$ with $r\leq h$. So $L=\frac{\sum_{i=1}^k2P_i+P_{i+k}}{3}+\frac{\sum_{j=1}^{2r}Q_j}{2}$, $M=\frac{\sum_{i=1}^kP_i+2P_{i+k}}{3}+\frac{\sum_{j=2r+1}^{2h}Q_j}{2}$ and 
		there exists a split non Galois triple cover of $\mathbb{P}^1$ totally branched on $\cup_{i=1}^{2k}P_i$ and simply branched on $\cup_{j=1}^{2h}Q_j$.
		The genus of the curve $C$, is given by  $2g(C)-2=-6+2(2k)+2h$ so $g(C)=2k+h-2\geq 1$.
		
		Now we consider the fiber product
		$$\xymatrix{ S\times_{\mathbb{P}^1}C\ar[r]\ar[d]&S\ar[d]^{\mathcal{E}}\\ C\ar[r]^{g}&\mathbb{P}^1}$$
		If the fibers of $\mathcal{E}$ over the points $P_i$ and $Q_j$ are smooth, the surface $X:=S\times_{\mathbb{P}^1}C$ is smooth and it is a triple non Galois cover of $S$ totally branched over $\cup_i\mathcal{E}^{-1}(P_i)$ and simply branched over $\cup_j\mathcal{E}^{-1}(Q_j)$.
		The fibration $\mathcal{E}$ induces a fibration $\mathcal{E}_X:X\ra C$ whose generic fiber is a smooth genus 1 curve and which has $2h$ fibers with multiplicity 2.
		
		It holds $h^{1,0}(X)\geq g(C)\geq 1$. The surface $X$ is necessarily proper elliptic, i.e. $\kappa(X)=1$, cfr. \ref{prop_kodairaDim} and \cite[Lemma III.4.6]{M89}.  
	\end{say}
	
	\begin{say}\label{say:non Galois split K3 Kod=2}{\bf A split non Galois triple cover of a K3 surface with a surface of general type.}
		Let $S$ be a K3 surface which admits an even set of $k$ disjoint rational curves $N_i$, so either $k=8$ or $k=16$. There exists a pseudo ample divisor $H$ which is contained in $\langle N_1,\ldots , N_k\rangle^{\perp_{NS(S)}}$ with $H^2=2h$ for a positive number $h$. 
		
		Then, putting $C'=3H$ $B'=0$, $B''=\sum_i N_i$ and $C''=0$, one obtains the data of a split non Galois triple cover: $$L=\frac{C'}{3}+\frac{B''}{2}=H+(\sum_{i=1}^{k}N_i)/2,\ \ M=\frac{2C'}{3}=2H.$$
		
		We consider the rank 2 vector bundle $\EE=\mathcal{O}(-L)\oplus \mathcal{O}(-M)$. Since $S^3(\EE^{\vee})\otimes \bigwedge^2\EE=\mathcal{O}(2L-M)\oplus \mathcal{O}(2M-L)\oplus \mathcal{O}(L)\oplus \mathcal{O}(M)$ admits global sections, $\EE$ is the Tschirnhausen bundle of a triple cover, see Theorem \ref{teo.miranda}.
		
		This triple cover, denoted by $f:X\ra S$ is totally ramified on $C'$ (i.e. one a curve contained in the linear system $|3H|$) and simply ramified on $\cup_i N_i$.

		With the notation of \eqref{diag split triple cover}, one has that $W$ is a non minimal surface and its minimal model is obtained by contracting $k$ $(-1)$curves. The minimal model of $W$ is a K3 surface or an Abelian surface according to the fact that $k=8$ or $k=16$. In particular, denoted by $E_i$ the $(-1)$-curves on $W$, one has $$K_W=-\sum_{i=1}^kE_i,\ \ K_WK_W=-k,\ \  \chi(\mathcal{O}_W)=\frac{16-k}{4},\ \ h^{1,0}(W)=\frac{k-8}{4},\ \ e(W)=48-2k.$$
		
		The Galois triple cover $\beta_2:Z\ra W$ is branched on a curve in the linear system $|\beta_1^*(3H)|$. Let us assume that the branch locus is a smooth curve in this linear system. Hence $\beta_2:Z\ra W$ is a smooth Galois triple cover, whose data are $(B,C,L,M)=(0,\beta_1^*C',\beta_1^*H,2\beta_1^*H)$ and whose invariant are 
		$$\begin{array}{ll}\vspace{0.2cm}\chi(\mathcal{O}_Z)=3\chi(\mathcal{O}_W)+\frac{1}{2}(\beta_1^*H)^2+\frac{1}{2}(2\beta_1^*H)^2=&\frac{48-3k}{4}+10h\\
			\vspace{0.2cm}
			K_Z^2=3K_W^2+2(\beta_1^*H)^2+2(2\beta_1^*H)^2+2(\beta_1^*H)^2=&-3k+48h\\
			\vspace{0.2cm}
			h^{i}(Z,\mathcal{O}_Z)=h^1(W,\mathcal{O}_W)+h^1(W,\beta_1^*H)+h^1(W,2\beta_1^*H)=&\frac{k-8}{4}\\
			e(Z)=3e(W)+4((\beta_1^*H)^2+(2\beta_1^*H)^2)-2(\beta_1^*H)^2=&144-6k+72h
		\end{array}
		$$
		We used that $\beta_1^*D\beta_1^*D=2D^2$ for every divisor $D\in Pic(S)$ and that $H$ is big and nef, so that $\beta_1^*H$ is big and nef and hence the vanishing theorems hold. Moreover, one obtains $$h^{2,0}(Z)=\chi(Z)-1+h^{1,0}(Z)=9-\frac{k}{2}+10h.$$
		
		Now $\alpha:Z\ra X$ is a double cover branched on $k$ rational curves: $\alpha^*(N_j)=A_j+2A_j'$, which are geometrically two copies of a rational curve $N_j$, but one has multiplicity 2, the other 1, only one of them is a component of the branch locus of $\alpha$.

		We want to apply the formulae \cite[Chapter 5 Section 22]{BHPV} to the double cover $Z\ra X$. The branch locus $J$ is such that  $-2k=(K_X+J)J$ by adjunction. This implies that $-k=K_XI+2I^2$ where $I$ is a divisor such that $2I=J$.

		By construction $f^{-1}(N_i)=M_i\cup M_i'$ where both $M_i$ and $M_i'$ are isomorphic to $N_i$, one of them has multiplicity 2 (because $N_i$ is contained in the simple ramification) and $M_i$ and $M_i'$ are disjoint. So $f^*(N_i)=M_i+2M_i'$. 
		
		Since $f$ is a $1:1$ map restricted to $M_i$ and $M_i'$, one obtains $f_*(M_i)=f_*(M_i')=N_i$ and, by the projection formula,
		$$M_i^2=(M_i+2M_i')M_i=f^*(N_i)M_i=N_if_*(M_i)=N_i^2=-2\Rightarrow M_i^2=-2$$
		$$2(M_i')^2=(M_i+2M_i')M_i'=f^*(N_i)M_i'=N_if_*(M_i')=N_i^2=-2\Rightarrow (M_i')^2=-1.$$
		The branch locus of the cover $\alpha$ consists of the curves $M_i$, i.e. with the previous notation $$J=\sum_{i=1}^kM_i\mbox{ so }I=\left(\sum_{i=1}^kM_i\right)/2\mbox{ and }I^2=-\frac{k}{2}.$$
		By $-k=K_XI+2I^2$ it follows that $K_XI=0$.
			
		Hence 
		
		$$\chi(Z)=\frac{48-3k}{4}+10h=2\chi(X)+\frac{1}{2}K_XI+\frac{1}{2}I^2=2\chi(X)-\frac{k}{4}$$
		$$K_Z^2=-3k+48h=2K_X^2+4K_XI+2I^2=2K_x^2-k$$
		$$e(Z)=144-6k+72h=2e(X)+2K_XI+4I^2=2e(X)-2k.$$
		
		So the invariants of the surface $X$ are $$e(X)=72+36h-2k,\ \ K_X^2=-k+24h,\ \  \chi(X)=(24-k)/4+5h.$$
		
		The surface $X$ is non minimal, since it contains at least $k$ $(-1)$-curves.
		We observe that in any case $K_X^2>0$ and $\kappa(X)\geq 0$, hence $X$ is of general type.
	
		We notice that $h^{2,0}(Z)\geq 11$ and by choosing $h$ big enough $h^{2,0}(Z)$ and $h^{2,0}(X)$ are arbitrarily big. 
		
	\end{say}

	\section{Triple covers of K3 surfaces: the Non Split Case}\label{sec: triple cover non split}
	
	The most general situation for a triple cover $f\colon X \ra S$ is when the Tschirnhausen bundle $\EE$ is indecomposable, in particular the cover is non Galois. The main question that one has to address is the existence of the Tschirnhausen bundle $\EE$; and this boils down to the study of rank 2 indecomposable vector bundle $\EE$ on a K3 surface which satisfy the further condition $H^0(Y, \, S^3 \EE^{\vee} \otimes \bigwedge^2\EE) \neq 0$ given in Theorem \ref{teo.miranda}.
	
	A standard approach (see e.g., \cite[Section 2]{PP13})  
	is to construct the Tschirnhausen bundle $\EE$ exploiting  the Cayley-Bacharach property (CB) of some 0-subscheme (see also \cite[Page 36]{F98} and \cite{Ca90}), which we recall for simplicity :
	
	\begin{theorem}\cite[Theorem 5.1.1]{HL}
		Let $Z \subset S$ be a local complete intersection of codimension two, and let $\cL$ and $\cM$ be line bundles on $S$. Then there exists an extension
		\begin{equation*}\label{eq_CB} 0\ra \mathcal{L} \ra \EE^{\vee}\ra \mathcal{M}\otimes \mathcal{I}_Z\ra 0
		\end{equation*}
		such that $\EE$ is locally free if and only if the pair $(\cL^{-1}\otimes \cM \otimes K_S,Z)$ has the Cayley-
		Bacharach property:
		\medskip 
		
		{\bf (CB)} If $Z' \subset Z$ is a subscheme with $\ell(Z') = \ell(Z) - 1$ and  $s \in H^0(S, \cL^{-1}\otimes \cM \otimes K_S)$ with $s|_{Z'}=0$, then $s|_Z=0$. 
	\end{theorem}

We consider 
	\begin{equation}\label{eq_CB} 0\ra \mathcal{L} \ra \EE^{\vee}\ra \mathcal{M}\otimes \mathcal{I}_Z\ra 0
	\end{equation}
where 
	$\mathcal{L}$ and $\mathcal{M}$ are line bundles on a K3 surface $S$ and $Z$ a 0-cycle.
	
	Notice that if $Z=\emptyset$ and the sequence \eqref{eq_CB} splits, then $\EE^{\vee} = \mathcal{L} \oplus \mathcal{M}$ and we are back to the cases treated in the previous sections. Therefore, we would like to assume that $Z \neq \emptyset$ and that the sequence \eqref{eq_CB} does not split.
	
	First we discuss some criteria which assure the existence of the triple cover associated to \eqref{eq: THE extension}, then we apply them to some possible choices  of the triple $(\mathcal{L},\mathcal{M},Z)$.
	
	The following proposition gives a conditions on $\mathcal{L}^{\vee}\otimes \mathcal{M}$ which assure the existence of the sequence \eqref{eq_CB}. 
	
	\begin{prop}\label{prop_exExt_Gen} If $h^0(S,\mathcal{L}^{\vee}\otimes \mathcal{M})=0$ and $h^1(S,\mathcal{L}^{\vee}\otimes \mathcal{M})\neq 0$, then  $\Ext^1(\mathcal{M}\otimes \mathcal{I}_Z, \mathcal{L}) \neq 0$ and the extension \ref{eq_CB} exists. In particular if $h^1(S,\mathcal{L}^{\vee}\otimes \mathcal{M})=1$ the extension is unique.
	\end{prop}

	\begin{proof} Let $\mathcal{G}:=\mathcal{L}$ and $\mathcal{F}=\mathcal{L}^{\vee}\otimes \mathcal{M} \otimes \mathcal{I}_Z$.
		We want to prove that $\Ext^1(\mathcal{F}\otimes \mathcal{G}, \mathcal{G}) \neq 0$.
		By Serre duality we have:
		\[ \Ext^1(\mathcal{F}\otimes \mathcal{G}, \mathcal{G})=\Ext^1(\mathcal{F}, \mathcal{O})=H^1(\mathcal{F})=H^1(\mathcal{L}^{\vee}\otimes \mathcal{M} \otimes \mathcal{I}_Z).
		\]
		Now consider the fundamental exact sequence of the scheme $Z$.
		\[ 0 \ra \mathcal{I}_Z \ra \mathcal{O}_S \ra \mathcal{O}_Z \ra 0
		\]
		Tensorised by $\mathcal{L}^{\vee}\otimes \mathcal{M}$ we get
		\[ 0 \ra \mathcal{I}_Z\otimes\mathcal{L}^{\vee}\otimes \mathcal{M}  \ra \mathcal{L}^{\vee}\otimes \mathcal{M}\ra \mathcal{O}_Z\otimes \mathcal{L}^{\vee}\otimes \mathcal{M} \ra 0.
		\]
		Since $H^0(S,\mathcal{L}^{\vee}\otimes \mathcal{M})=0$, one obtains $h^0(Z,\mathcal{O}_Z\otimes \mathcal{L}^{\vee}\otimes \mathcal{M})=0$
		and subsequently the long exact sequence in cohomology gives
		\[ 0 \ra H^1(\mathcal{I}_Z\otimes\mathcal{L}^{\vee}\otimes \mathcal{M}) \ra H^1(\mathcal{L}^{\vee}\otimes \mathcal{M}) \ra 0.
		\]
		So $$\dim\Ext^1(\mathcal{F}\otimes \mathcal{G}, \mathcal{G})=h^1(\mathcal{I}_Z\otimes\mathcal{L}^{\vee}\otimes \mathcal{M})=h^1(\mathcal{L}^{\vee}\otimes \mathcal{M})$$
		and the claim follows.
	\end{proof}
We observe that by Serre duality, on a K3 surface $$h^1(S,\mathcal{L}^\vee\otimes\mathcal{M})=h^1(S,\mathcal{L}\otimes \mathcal{M}^\vee),$$ so one can substitute the hypothesis $h^1(\mathcal{L}^\vee\otimes\mathcal{M})\neq 0$ with $h^1(\mathcal{L}\otimes \mathcal{M}^\vee)\neq 0$.

\begin{theorem}\label{theor: existence triple tschi} Let $S$ be a K3 surface, $Z$ a non empty 0-dimensional scheme on $S$, $\mathcal{L}$, $\mathcal{M}$ be two line bundles such that:\begin{itemize}\item $h^0(S,\mathcal{L}^{\vee}\otimes \mathcal{M})=0$;\item $h^1(S,\mathcal{L}^{\vee}\otimes \mathcal{M})=h^1(S,\mathcal{L}\otimes \mathcal{M}^{\vee})\neq 0$;\item  $h^0(S,\mathcal{L}^{\otimes 2}\otimes \mathcal{M}^{\vee})\geq 1$.\end{itemize} Then there exists a triple cover $X\ra S$ with Tschirnhausen $\mathcal{E}$, defined by the \eqref{eq_CB} for any possible choice of $Z$.

\end{theorem}
\begin{proof}
	The condition (CB) is automatically satisfied if $h^0(\mathcal{L}^\vee\otimes \mathcal{M})=0$ (see \cite[Theorem 12]{F98}) and by Proposition \ref{prop_exExt_Gen} $\EE$ exists and is locally free, and its dual as well.
	To assure the existence of the triple cover we have to prove that 
	
	\[
	h^0(S, \, S^3\EE^{\vee} \otimes \bigwedge^2 \EE ) \neq 0.
	\]
	
We apply the Eagon-Northcott complex (see e.g. \cite[Appendix 2]{E95} and \cite[Lemma 4.7]{CT07}) to the sequence \eqref{eq_CB} and we obtain
\[
0 \ra \mathcal{L} \otimes S^2\EE^{\vee}  \ra  S^3\EE^{\vee} \ra \mathcal{M}^3 \otimes  \mathcal{I}_Z^3 \ra 0.
\]	
Now, let us tensorize the previous sequence by $ \Lambda^2\EE \cong \mathcal{L}^{-1} \otimes \mathcal{M}^{-1}$ and we get
\[	
0 \ra S^2\EE^{\vee} \otimes \mathcal{M}^{-1} \ra  S^3\EE^{\vee} \otimes  \Lambda^2\EE\ra \mathcal{M}^2\otimes \mathcal{L}^{-1}  \otimes  \mathcal{I}^3_Z \ra 0.
\]	
So, if we prove that $S^2\EE^{\vee} \otimes \mathcal{M}^{-1}$ has global section we are done. To do so we apply Eagon-Northcott complex to the sequence \eqref{eq_CB} and we tensorize it by $\mathcal{M}^{-1}$. We have
\[
0 \ra \mathcal{L} \otimes \EE^{\vee} \otimes \mathcal{M}^{-1}  \ra  S^2\EE^{\vee} \otimes \mathcal{M}^{-1} \ra \mathcal{M}  \otimes \mathcal{I}_Z^2 \ra 0.
\]	
As a last step we  show that $\mathcal{L} \otimes \EE^{\vee} \otimes \mathcal{M}^{-1}$ has global section. This is true by hypothesis and by the short exact sequence 
\[
0\ra \mathcal{L}^2 \otimes  \mathcal{M}^{-1} \ra \EE^{\vee} \otimes \mathcal{L} \otimes  \mathcal{M}^{-1} \ra \mathcal{L}\otimes \mathcal{I}_Z\ra 0
\]
obtained  tensorizing the sequence \eqref{eq_CB} by $\mathcal{L} \otimes \mathcal{M}^{-1}$.

	Therefore, we have $h^0(S^3\EE^{\vee} \otimes \bigwedge^2 \EE) \geq 1$.

	\end{proof}
	
Now we give some explicit examples, choosing the line bundles $\mathcal{L}$ and $\mathcal{M}$.
Let $C$ be a smooth genus 1 curve on $S$, and hence a fiber of an elliptic fibration on $S$. We pose $$\mathcal{L}=\mathcal{O}_S(nC),\ \ \mathcal{M}=\mathcal{O}_S(mC).$$
If $-n+m<0$ then $h^0(\mathcal{L}^{\vee}\otimes\mathcal{M})=0$;

 if $n-m\geq 2$ then $h^1(\mathcal{L}\otimes \mathcal{M}^{\vee})\neq 0$ and  $h^0(\mathcal{L}^{\otimes 2}\otimes \mathcal{M})\geq 1$. Hence is $n\geq m+2$ the hypothesis of theorem \ref{theor: existence triple tschi} are satisfied and hence there exists the triple cover $X\ra S$.
If $n=m+2$, then the extension given by \eqref{eq_CB} is unique.	
	
So we now discuss the triple cover associated with the sequence 
	\begin{equation}\label{eq: THE extension}0\ra \oo_S(nC)\ra \EE\ra \mathcal{I}_Z(mC)\ra 0\ \ \ n\geq m+2\end{equation}
and in particular we analyze it according to the choice of 0-scheme $Z$.

		\begin{lemma}\label{lemma: l(Z)=1}
		Let $S$ be an elliptic K3 surface with elliptic fibration $\varphi_{|C|}:S\ra\mathbb{P}^1$. 
		Let us assume that $Z$ is a 0-cycle on  $S$ such that $\ell(Z)=1$. 
		\begin{itemize}
			\item If $m \geq 2$ then $h^0\big(\mathcal{I}_Z(mC)\big)\neq 0$ and $h^1\big(\mathcal{I}_Z(mC)\big)\neq 0$. 
			\item If $m =1$ then $h^0\big(\mathcal{I}_Z(C)\big)=1,$ and $h^1\big(\mathcal{I}_Z(C)\big)=0$.					\end{itemize}
	\end{lemma}
	\proof Since $\ell(Z)=1$, the subscheme $Z$ consists in a single point $p$. 
	By
	\[ 0 \ra \mathcal{I}_p\big(mC\big) \ra \mathcal{O}_S\big(mC\big) \ra \mathcal{O}_p\big(mC\big) \ra 0
	\] one obtains the long exact sequence
	\begin{equation}\label{eq_longLz=1} \begin{split}
	0\ra H^0\left(\mathcal{I}_p(mC)\right) \ra H^0\left(\mathcal{O}_S(mC)\right)& \ra  H^0\big(\mathcal{O}_p(mC)\big) \ra \\
	\ra H^1(\mathcal{I}_p\big(mC\big)) \ra H^1(\mathcal{O}_S\big(mC\big)) &\ra  0,
	\end{split}
	\end{equation} which is
	\[ 0\ra H^0\left(\mathcal{I}_p(mC)\right) \ra \C^{m+1}\ra \C \ra H^1(\mathcal{I}_p\big(mC\big)) \ra \C^{m-1}\ra 0.
	\]
	This yields at once the first statement. For the second one, let us insert the value $m=1$ in \eqref{eq_longLz=1} and get 
	\[ 0\ra H^0\left(\mathcal{I}_p(C)\right) \ra H^0\left(\mathcal{O}(C)\right)\stackrel{\psi}{\ra} H^0\big(\mathcal{O}_p(C)\big) \cong \C \ra H^1(\mathcal{I}_p\big(C\big)) \ra  0.
	\]
	By \cite[Proposition 3.10]{H} the linear system $|C|$ is base point free, hence $\psi$ -- which is an evaluation map --  is not the zero map and we have conclude the proof. 
\endproof
	
	\begin{lemma}\label{lemma: l(Z)=2}
		Let $S$ be an elliptic K3 surface with elliptic fibration $\varphi_{|C|}:S\ra\mathbb{P}^1$. 
		Let us assume that $Z$ is a 0-cycle on  $S$ such that $\ell(Z)=2$. 
		\begin{itemize}
			\item If $m \geq 2$ then $h^0\big(\mathcal{I}_Z(mC)\big)\neq 0$ and $h^1\big(\mathcal{I}_Z(mC)\big)\neq 0$. 
			\item If $m =1$ then $h^0\big(\mathcal{I}_Z(C)\big)=h^1\big(\mathcal{I}_Z(C)\big)=0$ if $Z$ consists of $2$ distinct smooth points $z_1$ and $z_2$   which do not lie on the same fiber of the fibration. 
			\item  If $m =1$ then $h^0\big(\mathcal{I}_Z(C)\big)=h^1\big(\mathcal{I}_Z(C)\big)=1$ if $Z$ consists of $2$ distinct smooth points $z_1$ and $z_2$   which lie on the same fiber of the fibration or $Z$ is a single point. 
		\end{itemize}
	\end{lemma}
	\proof The first statement is proven exactly in the same way as in Lemma \ref{lemma: l(Z)=1}.

	Let  $m=1$ then $H^0(\oo_S(C)) \cong \CC^2$ and also $H^0\big(\mathcal{O}_Z(C)\big) \cong \CC^2$.  
	In the long exact sequence
	\begin{equation*}
			0\ra H^0\left(\mathcal{I}_Z(C)\right) \ra H^0\left(\mathcal{O}_S(C)\right) \stackrel{\psi}{\ra}  H^0\big(\mathcal{O}_Z(C)\big) \ra  H^1(\mathcal{I}_Z\big(C\big)) \ra H^1(\mathcal{O}_S\big(C\big)) =  0,
	\end{equation*} 
	the map $\psi$ is the evaluation map $ev\colon s \mapsto s(x)$ with $x \in Z$, which is the zero map if and only is $Z$ is in not the base locus of $|C|$. By \cite[Proposition 3.10]{H} the linear system $|C|$ is base point free. This yields that $h^0(\ii_Z(C)) \leq 1$. 

	There are two cases:  
	\begin{enumerate}
		\item there is a section $s \in H^0(\oo_S(C))$ which passes through $Z$, and in this case  $Z$ consists either of $2$ distinct smooth points $z_1$ and $z_2$   which lie on the same fiber of $\varphi_{|C|}$ or $Z$ is a single double point. 
		In this case 
		\[
		H^0\left(\mathcal{I}_Z(C)\right) \cong <s>.
		\]
		\item no single section passes through $Z$ and in this case $Z$ consists of $2$ distinct smooth points $z_1$ and $z_2$   which do not lie on the same fiber of the fibration and $H^0\left(\mathcal{I}_Z(C)\right)=0$.  
	\end{enumerate}
	\endproof

	\begin{lemma} 
		The total Chern classes of the vector bundle $\EE$, determined by the extension \eqref{eq: THE extension}, is $c(\EE)=(1,(n+m)C,\ell(Z))$.
	\end{lemma}
	\proof We use $ch(C\otimes\mathcal{I}_Z)=ch(C)ch(\mathcal{I}_Z).$
	
	Recalling that 
	\[ ch(V)=(rk(V),c_1(V),\frac{1}{2}\left(c_1^2(V)-2c_2(V)\right),
	\] one has
	$ch(C)=(1,C,0)$, $ch(\mathcal{I}_Z)=(1,0,-\ell(Z))$ and thus $$ch(C\otimes\mathcal{I}_Z)=(1,C,-\ell(Z)),\mbox{ so }c(C\otimes\mathcal{I}_Z)=(1,C,\ell(Z)).$$
	Since $c(\EE)=c(nC)c(mC\otimes\mathcal{I}_Z)$ (see e.g., \cite[Section 5]{HL}), one obtains 
	\[ c(\EE)=(1,nC,0)(1,mC,\ell(Z))=(1,(n+m)C,\ell(Z)).
	\]
	\endproof

	\begin{lemma} \label{lemma (n,m,l)=(3,1,2)} Let $S$ be an elliptic K3 surface with elliptic fibration $\varphi_{|C|}:S\ra\mathbb{P}^1$. 

	Let $(n,m,\ell(Z))=(3,1,2)$.
	
	If $Z$ is supported on two points $z_1$, $z_2$  which do not lie on the same fiber of the fibration $\varphi_{|C|}:S\ra\mathbb{P}^1$, then $h^0(\EE)=4$ and $h^1(\EE)=2$.

	\end{lemma}
	\proof Suppose  $(n,m,\ell(Z))=(3,1,2)$. The computation of $h^i(\EE)$ is based on the sequence
	$$0\ra H^0(3C) \cong \CC^4 \ra H^0(\EE)\ra H^0(\mathcal{I}_{Z}(C))\ra H^1(3C) \cong \CC^2 \ra H^1(\EE)\ra H^1(\mathcal{I}_Z(C)) \ra 0$$

		If $Z$ consists of $2$ distinct smooth points $z_1$ and $z_2$   which do not lie on the same fiber of the fibration then by Lemma  \ref{lemma: l(Z)=2} we have $H^0\big(\mathcal{I}_{Z}(C)\big)=H^1\big(\mathcal{I}_{Z}(C)\big)=0$ and we have $\left(h^0(\EE),h^1(\EE)\right)=\left(4,2\right)$.
			\endproof
		
	\begin{prop} Let us suppose that  $S$ is an elliptic K3 surface with elliptic fibration $\varphi_{|C|}:S\ra\mathbb{P}^1$ and that $(n,m,\ell(Z))=(3,1,2)$.  Moreover let us assume that $Z$ is supported on two points $z_1$, $z_2$  which do not lie on the same fiber of the fibration $\varphi_{|C|}:S\ra\mathbb{P}^1$. Then there exists a properly elliptic surface $X$ which is a non Galois triple cover of an elliptic K3 surface, such that $(h^{1,0}(X),h^{2,0}(X))=(3,6)$.

	\end{prop}
	\begin{proof}
		
		The existence of the triple cover follows by Theorem \ref{theor: existence triple tschi}, hence the surface $X$ exists. Moreover, the birational numerical invariants of $X$ are given by Proposition \ref{prop.invariants} $(i)$ with the information given in \ref{lemma (n,m,l)=(3,1,2)}.
		
		By Proposition \ref{thm_Branch} the branch divisor of $f\colon X \ra S$ is given by $\Lambda^2\EE^{-2}\simeq  \oo_S(8C)$. Finally, by Proposition \ref{prop_kodairaDim} $X$ is a properly elliptic surface. 
	\end{proof}
\begin{rem}{\rm
	The triple cover $X\ra S$ is not induced by a base change $g:C\ra\mathbb{P}^1$ (for a certain curve $C$) as in Section \ref{say:non Galois split K3 Kod=1} and in Proposition \ref{prop: examples case 3 of proposition} because otherwise $\EE$ would split.}
	\end{rem}

Another possible choice of $\mathcal{L}$ and $\mathcal{M}$ in the sequence \eqref{eq_CB} is presented in the following corollary.

 \begin{corollary}
	Let $S$ be a K3 surface with an irreducible curve $C$ of selfintersection $2d>0$ such that there exists $h$ disjoint rational curves $R_i$'s, which are disjoint also from to $C$. If $h\leq 10$ a K3 surface with this configuration of curves exists. 
	
	If $2\leq h\leq 10$ and $(9h-1)/4\leq d\leq 4h-3$ there exists a non Galois triple cover $f:S\ra X$ whose Tschirnhausen is determined by the sequence:
	$$0\ra \oo_S\left(C-\sum_{i=1}^hR_i\right)\ra \EE^{\vee}\ra \mathcal{I}_Z\left(\sum_{i=1}^hR_j\right)\ra 0\ \ \ $$
	The surface $X$ is a surface of general type.
	
\end{corollary} 	
\begin{proof}
	The existence of $S$ depends on the existence of a primitive embedding of the lattice spanned by $C$ and the $R_i$ in the K3 lattice, which is guaranteed if $h\leq 10$, because in this case the lattice has rank at most 11.
	
The genus of the curve $C$ is $g(C)=d+1>1$, by hypothesis. We pose $$\mathcal{L}=\mathcal{O}_S\left(C-\sum_{i=1}^hR_i\right),\ \ \  \mathcal{M}=\mathcal{O}_S\left(\sum_{i=1}^h R_i\right).$$
 Since $\mathcal{L}^{\vee}\otimes \mathcal{M}=\mathcal{O}_S\left(-C+2\sum_{i=h}^hR_i\right)$, one has $h^0(\mathcal{L}^{\vee}\otimes \mathcal{M})=0$. Indeed $C$ is an irreducible divisor with positive self intersection and hence it would intersect non negatively every effective divisor, but $(-C+2\sum_{i=h}^hR_i)C=-CC<0$.
 
 Moreover, $(-C+2\sum_{i=h}^hR_i)^2=2d-8h$ and if $ d\leq 4h-3$, $(-C+2\sum_{i=h}^hR_i)^2\leq -6$. By Riemann Roch theorem one obtains $h^1(\mathcal{L}^{\vee}\otimes \mathcal{M})\neq 0$. Moreover
$$\mathcal{L}^{\otimes 2}\otimes \mathcal{M}^{\vee}=\mathcal{O}_S\left(2C-3\sum_{i=1}^hR_i\right).$$ 

If $d\geq (9h-1)/4$, $\left(2C-3\sum_{i=1}^hR_i\right)^2=8d-18h\geq -2$ and since $\left(2C-3\sum_{i=1}^hR_i\right)C>0$, one has $h^0(\mathcal{L}^{\otimes 2}\otimes \mathcal{M}^{\vee})\geq 1$.
We conclude that the triple cover exists by Theorem \ref{theor: existence triple tschi}.
 
 Since the branch divisor is $\wedge^2\EE^{-2}=(\mathcal{L}\otimes \mathcal{M})^{\otimes 2}=\mathcal{O}_S(2C)$, if the branch locus were not reduced, it would be 2 times $C_q$ where $C_q$ is a specific curve in $|C|$.
 Moreover every global section of $S^3\EE\otimes \bigwedge^2\EE^{\vee}$ should vanish along $C_q$. In particular also every global section of $\mathcal{L}^{2}\otimes \mathcal{M}^{\vee}$ would vanish along $C_q$. But this implies that $$H^0(\mathcal{O}_S(2C-3\sum R_i))=H^0(\mathcal{L}^{2}\otimes \mathcal{M}^{\vee})=H^0(\mathcal{L}^{2}\otimes \mathcal{M}^{\vee}\otimes\mathcal{O}_S(-C_q))=H^0(\mathcal{O}_S(C-3\sum R_i)),$$ which is not the case.
 
 In particular $X$ is normal and the branch locus contains curves with positive genus, so $X$ is of general type.
 
 \end{proof}

	\bigskip
	
	Alice Garbagnati  Universit\`a degli Studi di Milano,  Dipartimento di Matematica \emph{''Federigo Enriques"}, I-20133 Milano, Italy \\
	\emph{e-mail} \verb|alice.garbagnati@unimi.it|
	
	\bigskip
	
	Matteo Penegini, Universit\`a degli Studi di Genova, DIMA Dipartimento di Matematica, I-16146 Genova, Italy \\
	\emph{e-mail} \verb|penegini@dima.unige.it|
	
	\medskip

\end{document}